\documentclass[a4paper,11pt,reqno]{amsart}
\usepackage{amsmath,amsfonts,amsthm,amssymb}
\usepackage{dsfont}
\usepackage[foot]{amsaddr}
\usepackage{comment}
\usepackage{enumitem}
\usepackage{hyperref}
\usepackage{graphicx,color}
\usepackage{srcltx}
\setlist[itemize]{wide=0pt, leftmargin=0pt}
\theoremstyle{plain}
\newtheorem{lem}{Lemma}[section]
\newtheorem{thm}[lem]{Theorem}
\newtheorem{prop}[lem]{Proposition}
\newtheorem{cor}[lem]{Corollary}
\theoremstyle{definition}

\newtheorem{defn}[lem]{Definition}
\newtheorem{rem}[lem]{Remark}

\numberwithin{equation}{section}
\newcommand{\N}{\mathbb{N}}

\newcommand{\R}{\mathbb{R}}

\newcommand{\cala}{\mathcal A}

\newcommand{\calm}{\mathcal M}

\newcommand{\cals}{\mathcal S}
\newcommand{\calt}{\mathcal T}

\newcommand{\norm}[1]{\left\Vert#1\right\Vert}

\newcommand{\eps}{\varepsilon}

\newcommand{\dist}{\operatorname{dist}}

\renewcommand{\geq}{\geqslant}
\renewcommand{\leq}{\leqslant}

\newenvironment{nouppercase}{%
	\renewcommand{\uppercasenonmath}[1]{}}{}
\begin{document}
\title{$C^1$ invariant, stable and inertial manifolds  for non-autonomous dynamical systems}

\author[R. Czaja]{Rados{\l}aw Czaja$^{1}$}

\address{$^1$Institute of Mathematics, University of Silesia in Katowice, Bankowa 14, 40-007 Katowice, Poland.
E-mail address: \textup{radoslaw.czaja@us.edu.pl}}

\author[P. Kalita]{Piotr Kalita$^{2,*}$}

\address{$^2$Faculty of Mathematics and Computer Science, Jagiellonian University, {\L}ojasiewicza 6, 30-348 Krak\'ow,
Poland. E-mail address: \textup{piotr.kalita@uj.edu.pl}}

\thanks{$^*$Corresponding author.}

\thanks{The work of Piotr Kalita was supported by FAPESP, Brazil grant 2020/14075-6 and Spanish Ministerio de Ciencia, Innovación y Universidades, Agencia Estatal de Investigación (AEI) and FEDER grant PID2024-156228NB-I00.}

\author[A. Oliveira-Sousa]{Alexandre N. Oliveira-Sousa$^{3}$}

\address{$^{3}$Departamento de  Ma\-te\-m\'{a}\-ti\-ca\\ Universidade Federal de Santa Catarina \\	Caixa Postal 88040-900, Florian\'opolis SC, Brazil. E-mail address:
\textup{alexandre.n.o.sousa@ufsc.br}}



\begin{abstract}

We use the version of the Lyapunov--Perron method operating on individual solutions to investigate the existence of invariant manifolds for non-autonomous dynamical systems, focusing in particular on inertial and stable manifolds. We establish a characterization of both types of manifolds in terms of solutions exhibiting a common growth behavior, analogous to the classical characterization involving hyperbolicity. Furthermore, we introduce a unified formulation of the gap condition, from which known sharp versions are derived. Finally, we show that the constructed inertial manifolds have $C^1$ regularity.

\end{abstract}
\begin{nouppercase}
\maketitle
\end{nouppercase}

\section{Introduction}
We consider the non-autonomous abstract parabolic problem governed by the equation
$$
u_t=A(t)u+f(t,u),
$$
defined on a possibly infinite-dimensional Banach space $X$, where $\{A(t)\colon t\in \R\}$ is a time dependent family of linear operators, and $f(t,\cdot)$ are Lipschitz functions. The equation defines a nonlinear process $\{T(t,\tau)\colon t\geq \tau\}$ of mappings $T(t,\tau)\colon X\to X$ which, for the initial value taken at time $\tau$, return the solution of the problem at time $t$. 

An \textit{invariant manifold} for this system is the time dependent family of sets 
$$
\calm(t) = \{ \phi + \Sigma(t,\phi)\colon \phi\in Q(t)X\},
$$
where the functions $\Sigma(t,\cdot)$ are Lipschitz with common Lipschitz constant, $\{Q(t)\colon t\in\R\}$ is a~family of projections related to operators $A(t)$, and we have the invariance relation $T(t,\tau)\calm(\tau)=\calm(t)$ for $t\geq \tau$. If this manifold is, in addition, exponentially attracting, it is called an \textit{inertial manifold}, see Definition~\ref{defn:INERTIAL}. 
Once we have constructed an invariant manifold for a given problem we can reduce its dimensionality  by restricting the dynamics to  the manifold, the resulting system being invertible, and, if this manifold is inertial, all solutions starting outside of it are attracted towards it. 

There are several methods used to construct inertial manifolds: two most widely used are the Hadamard's graph transformation method  and the Lyapunov--Perron method. The first one, more geometrical in nature, was applied to construct the inertial manifolds by Mallet-Paret and Sell \cite{MPS}, and later was extended to the non-autonomous situation by Koksch and Siegmund \cite{KS02}. The Lyapunov--Perron method, on the other hand, was used by D. Henry \cite{Henry} to construct stable, unstable and center manifolds for parabolic problems, and, in the work of Foias, Sell, and Temam \cite{FST}, it was adapted for construction of inertial manifolds. We note that in \cite{FST} the concept of inertial manifolds was introduced, and their existence was proved for the first time. Since their introduction, the theory of invariant and inertial manifolds has greatly developed: a very good overview of known results, as of 2014, can be found in \cite{Zelik}. 

Existence of invariant manifolds for a given problem requires a condition known as \textit{spectral gap} which bounds from above the Lipschitz constant of $f(t,\cdot)$ by a quantity dependent on the difference of two exponents coming from the \textit{exponential splitting} (Definition~\ref{def:splitting}) of the linear evolution process related to the linear equation $u_t=A(t)u$. The natural question is to derive the sharp, or optimal spectral gap condition. This was done by Miklav\v{c}i\v{c} \cite{Miklavcic} and Romanov \cite{Romanov}, who, in autonomous framework, constructed inertial manifolds using the sharp gap condition and gave examples from which it follows that the condition cannot be further improved. 

In non-autonomous framework \cite{KS02} used the Hadamard's graph transform approach construct the inertial manifold which becomes the time-parameterized family of Lipschitz graphs. Recently, in the work \cite{CLMO-S25}, which serves as a direct inspiration for the present article, the authors employed the Lyapunov-–Perron method to construct non-autonomous Lipschitz invariant and inertial manifolds. A key idea in their approach is that exponential splitting for the linear process is preserved under nonlinear perturbations, an idea we also further develop here. Furthermore, by introducing the notion of a stable manifold associated with an invariant manifold, \cite{CLMO-S25} establishes a link between the saddle-point property and invariant manifolds within the non-autonomous framework.

In the present work, we further extend the invariant manifold theory in non-autonomous framework. In particular, we refine the gap condition from \cite{CLMO-S25} used in the construction of Lipschitz invariant, inertial, and stable manifolds, and also obtain improved associated Lipschitz constants.
In addition, we establish new characterizations of these manifolds in terms of spaces of solutions that exhibit controlled growth, where the growth rate explicitly depends on the gap condition. This approach draws parallels with classical results concerning exponential dichotomies and the saddle-point property, see, for example \cite{CCLO-S-AA-21,Chen-Tan-Hale,Zhou-Lu-Zhang-13} where such characterizations are established. This further supports the central message of \cite{CLMO-S25}, emphasizing the strong interrelation among all these concepts.

In \cite{CLMO-S25}, the gap condition is formulated in terms of the rate between the spectral gap and the size of the nonlinearity (denoted by $(\gamma - \rho)/\ell$ in \cite{CLMO-S25}). In the present work, we propose a new and  more general framework for interpreting the gap condition; see Definition~\ref{def:gap}. Building upon ideas from \cite{CR96,LL99}, we refine the condition in \cite{CLMO-S25}, obtaining a version that, when expressed in suitable norms, coincides with the sharp condition established in \cite{Miklavcic,Romanov}; see Definition~\ref{def:gap} and Remark~\ref{rem:SHARPNESS}.

Following \cite{CR96,LL99}, instead of considering the Lyapunov--Perron fixed point operator which transforms Lipschitz graphs to Lipschitz graphs, we study a version of this operator that acts on continuous functions of time having appropriate growth, at $-\infty$ for the invariant manifold and at $+\infty$ for the stable manifold. This differs from the approach of \cite{CLMO-S25} as it operates on individual fibers rather than on the whole graphs. Although this method does not yield automatically the Lipschitzness of the constructed manifold (we obtain this property using the \textit{cone condition} and later improve the Lipschitz constant in an iterative process), the advantage of this approach is the possibility of obtaining the sharp gap. We note that in the non-autonomous case the same gap as ours has been obtained using the Hadamard's method in \cite{KS02}, and, moreover, \cite[Section 6]{LL99} discusses the extension of the method to the case of cocycles which entails the non-autonomous framework. Still, our exposition fits into the unified framework of \cite{CLMO-S25} which allows for the extension to study the stable manifold of an invariant manifold which in turn leads to various applications in autonomous and non-autonomous frameworks such as, for example, the study of robustness of exponential dichotomies. The study of these topics under the assumption of our gap condition from Definition~\ref{def:gap} as a continuation of this article. Here, in addition, we prove that under the same gap condition the constructed invariant and inertial manifolds are $C^1$ smooth. 
This smoothness of the manifold is important in  further applications as it allows us to construct the tangent space to the manifold at every point, and moreover it implies that the reduced
vector field on the inertial manifold also inherits its smoothness. 

As a possible extension of this work we remark the problem of computation of non-autonomous inertial manifolds. A way to achieve this is using the ideas of \cite{JRT} where an algorithm for an autonomous case have been developed. 
 
The structure of the article is as follows: Section \ref{sec:2} contains  the setup of the studied problem and presents and discusses the assumptions needed for the invariant manifold existence. The main result, Theorem~\ref{thm:MAIN1}, on the existence of invariant manifolds is contained and proved in Section~\ref{sec:LIP}. Finally, Sections \ref{sec:diff} and \ref{sec:c1} contain proofs that the constructed manifolds are differentiable and $C^1$ (see Theorems~\ref{thm:MAIN2} and \ref{thm:MAIN3}), respectively. 

\section{Problem setup and assumptions}\label{sec:2}

\subsection{Linear evolution process and its exponential splitting.} Let $(X,\norm{\cdot})$ be a Banach space. We denote by $\mathcal{L}(X)$ the space of linear and bounded operators from $X$ to itself. For a possibly unbounded linear operator $A$ in $X$ we denote by $D(A)$ its domain. Moreover, let $J=\{ (s,t)\in \R^2\colon s\geq t\}$.

\begin{defn}
A family of linear and bounded operators $\{L(t,\tau)\colon \ (t,\tau)\in J\}\subset\mathcal{L}(X)$ is called a \textit{linear evolution process} if
	\begin{itemize}
		\item[(i)] $L(t,\tau) = L(t,s)L(s,\tau)$ for every $t\geq s\geq \tau$;
		\item[(ii)] $L(t,t) = I$ for every $t\in\R$;
		\item[(iii)] the function $J\ni (t,\tau)\mapsto L(t,\tau)\eta\in X$ is continuous for each $\eta\in X$.
	\end{itemize} 
\end{defn}
  We consider a family of (possibly unbounded) linear operators $A(t)\colon X\supset D(A(t))\to X$ defined for $t\in\R$. 
We assume that the family $\{A(t)\colon t\in\R\}$ generates a linear evolution process $\{L(t,\tau)\colon \ (t,\tau)\in J\}$ such that for each $(\tau,\eta)\in\R\times X$ the function $u(t,\tau,\eta)=L(t,\tau)\eta$ for $t\geq\tau$ is a mild solution of the abstract Cauchy problem   
\begin{equation}\label{e:LACP}
\begin{cases}
u_t=A(t)u\ \ \textrm{for}\ \  t>\tau,\\
u(\tau)=\eta\in X.
\end{cases}
\end{equation}
\begin{defn}\label{def:splitting}
	We say that the linear evolution process $\{L(t,\tau)\colon \ (t,\tau)\in J\}$ has exponential splitting with a bound $M\geq 1$,  exponents $\gamma,\rho\in\R$ such that $\gamma>\rho$ and a family of projections $\{Q(t)\colon t\in\R\}\subset\mathcal{L}(X)$ if
	\begin{itemize}
		\item[(i)] for every $(t,\tau)\in J$ we have \begin{equation*}
			Q(t)L(t,\tau)=L(t,\tau)Q(\tau);
		\end{equation*}
		\item[(ii)] $L(t,\tau)|_{Q(\tau)X}\colon Q(\tau)X\to Q(t)X$  is an isomorphism with the inverse denoted by $$L(t,\tau)^{-1} = L(\tau,t)\colon Q(t)X\to Q(\tau)X\ \ \text{for every}\ \ (t,\tau)\in J;$$
		\item[(iii)] the following  estimates hold
		\begin{align}\label{e:QTST}
			& e^{\gamma(t-\tau)}\norm{L(t,\tau)(I-Q(\tau))}\leq M\ \ \text{for}\ \  (t,\tau)\in J,\\
		\label{e:QTUNST}
		&	e^{\rho(t-\tau)}\norm{L(t,\tau)Q(\tau)}\leq M\ \ \text{for}\ \  (\tau,t) \in J.
	\end{align}		 
	\end{itemize}
\end{defn}

If $\gamma>0>\rho$, that is, $-\rho>0>-\gamma$, then $Q(t)$ is a projection of the unstable
subspace of $X$ at time $t$, and the complementary projection
 $I-Q(t)$ is the projection on the stable subspace of $X$ at time $t$. In the special case $\rho = -\gamma$ we say that the linear evolution process has \textit{exponential dichotomy} with constants $M\geq 1$ and $\gamma>0$. 
 
 \subsection{Renorming of the space \texorpdfstring{$X$}.} For every $t\in \R$ the space $X$ is a sum of two linear subspaces $Q(t)X$ and $(I-Q(t))X$. We introduce the following families of time dependent norms on these subspaces.
\begin{defn}
We define the following nonnegative valued functions: 
\begin{align}&\label{e:NORMNT}
|x|_{N(\tau)}=\sup_{t\leq\tau}e^{\rho(t-\tau)}\norm{L(t,\tau)x}\ \ \text{for}\ \  x\in N(\tau)=Q(\tau)X,
\\ &\label{e:NORMST}
|x|_{S(\tau)}=\sup_{t\geq\tau} e^{\gamma(t-\tau)}\norm{L(t,\tau)x}\ \ \textrm{for}\ \ x\in S(\tau)=(I-Q(\tau))X.
\end{align}
\end{defn}
The next result is a direct consequence of the above definition, see \cite[Lemma 6.8]{CLR}. 
\begin{lem}\label{lem:norms}
If the process  $\{L(t,\tau)\colon (t,\tau)\in J\}$ has exponential splitting, then
the functions $|\cdot|_{N(\tau)}$ and $|\cdot|_{S(\tau)}$ are norms in $N(\tau)$ and $S(\tau)$, respectively, which are equivalent to $\norm{\cdot}$ on each of the spaces. Moreover, we have
\begin{equation}\label{e:EQUIVNS}
\norm{x}\leq|x|_{N(\tau)}\leq M\norm{x}\ \text{for}\ x\in N(\tau)\quad\text{and}\quad  \norm{x}\leq|x|_{S(\tau)}\leq M\norm{x}\ \text{for}\ x\in S(\tau)
\end{equation}
and
\begin{equation}\label{e:QI-QONX}
|Q(\tau)x|_{N(\tau)}\leq M\norm{x}\ \text{and}\  |(I-Q(\tau))x|_{S(\tau)}\leq M\norm{x}\ \text{for}\  x\in X.
\end{equation}
Furthermore, we have
$$e^{\rho t}|L(t,\tau)x|_{N(t)}\leq e^{\rho\tau}|x|_{N(\tau)}\ \text{for}\  t\leq\tau \ \text{and}\  x\in N(\tau),$$
$$e^{\gamma t}|L(t,\tau)x|_{S(t)}\leq e^{\gamma\tau}|x|_{S(\tau)}\ \text{for}\  t\geq\tau\ \text{and}\  x\in S(\tau).$$
\end{lem}
The last two inequalities mean that after renorming, the estimates \eqref{e:QTST} and \eqref{e:QTUNST} hold with $M=1$. Thus, the choice of the norms in a way adjusted to the operator of the problem allows us to avoid the constant $M$ in the bounds of the exponential splitting. 

We choose a norm $\Gamma(\cdot,\cdot)$ on $\R^2$ that we will use to build the new norm on $X$ from the norms $|\cdot|_{N(\tau)}$ and $|\cdot|_{S(\tau)}$. This $\Gamma(\cdot,\cdot)$ is assumed to satisfy the following property. 

\begin{defn}\label{defn:ADMISSIBLENORM}
We say that a norm $\Gamma\colon\R^{2}\to[0,\infty)$ is admissible if for every $a,b\geq 0$ functions $\Gamma(a,\cdot)$ and $\Gamma(\cdot,b)$ are strictly monotone on $[0,\infty)$. 
\end{defn}

Among admissible norms on $\R^2$ we distinguish the norms  $\norm{(a_1,a_2)}_{p}=(|a_1|^p+|a_2|^p)^\frac{1}{p}$ with $p\geq 1$ or $\norm{(a_1,a_2)}_{\infty}=\max\{|a_1|,|a_2|\}$. Due to the equivalence of norms, we introduce a constant $c_\Gamma>0$ such that
\begin{equation*}
|a_1|+|a_2|=\norm{(a_1,a_2)}_{1}\leq c_\Gamma \Gamma(a_1,a_2)\ \text{for every}\ (a_1,a_2)\in\R^{2}.
\end{equation*}
We have $c_\Gamma=2^{1-\frac{1}{p}}$ if $\Gamma(a_1,a_2)=\norm{(a_1,a_2)}_{p}$, $p\geq 1$, and $c_\Gamma=2$ if $\Gamma(a_1,a_2)=\norm{(a_1,a_2)}_{\infty}$.

The next result follows in a straightforward way from Lemma \ref{lem:norms}.
\begin{lem}\label{lem:EQUIVNORMS}
Given an admissible norm on $\R^2$, for each $\tau\in\R$ the function
\begin{equation}\label{e:MOVINGNORM}
\norm{x}_\tau=\Gamma\left(|Q(\tau)x|_{N(\tau)},|(I-Q(\tau))x|_{S(\tau)}\right)\ \text{for}\ x\in X
\end{equation}
is an equivalent norm in $X$. We have 
$$\frac{1}{c_{\Gamma}}\norm{x}\leq\norm{x}_{\tau}\leq M\Gamma(1,1)\norm{x}\text{  for }x\in X\ \text{and}\ \tau\in\R.$$
\end{lem}

\subsection{Nonlinearity, nonlinear process, and the gap condition.} Let $f\colon\R\times X\to X$ be a~continuous function such that 
\begin{equation}\label{e:ZEROCONDITION}
f(t,0)=0\text{ for all }t\in\R.
\end{equation}
We need the nonlinearity $f$ to be Lipschitz with respect to its second argument. Specifically, it is convenient for us to make the Lipschitz assumption separately for $Q(t)$ projection of $f$ and for its complement. Namely, the standing assumption of this paper is that there exist $L_1,L_2>0$ such that for $t\in\R$ and $u,v\in X$ we have
\begin{align}\label{e:L1}
&|Q(t)(f(t,u)-f(t,v))|_{N(t)}\leq L_1\norm{u-v}_{t},\\
&\label{e:L2}
|(I-Q(t))(f(t,u)-f(t,v))|_{S(t)}\leq L_2\norm{u-v}_{t}.
\end{align}
Note that if  $f$ is uniformly globally Lipschitz in the second variable with the Lipschitz constant $\ell>0$, i.e.,
\begin{equation}\label{e:GLOBALLIP}
\norm{f(t,u)-f(t,v)}\leq\ell\norm{u-v},\ u,v\in X,\ t\in\R,
\end{equation} 
then assumptions \eqref{e:L1}, \eqref{e:L2} hold with $L_1=L_2=M\ell c_{\Gamma}$.

We consider an abstract Cauchy problem 
\begin{equation}\label{e:PERTURBED}
\begin{cases}
u_t=A(t)u+f(t,u)\ \ \textrm{for}\ \ \ t>\tau,\\
u(\tau)=\eta\in X,
\end{cases}
\end{equation}
which generates a nonlinear evolution process $\{T(t,\tau)\colon (t,\tau)\in J\}$ in $X$, that satisfies
\begin{equation}\label{e:VCF}
T(t,\tau)\eta=L(t,\tau)\eta+\int_{\tau}^{t}L(t,s)f(s,T(s,\tau)\eta)ds\ \ \text{for}\ \ t\geq\tau,\ \eta\in X.
\end{equation}
Thus $T(t,\tau)\colon X\to X$ satisfies $T(t,t)=I$, $T(t,s)T(s,\tau)=T(t,\tau)$ for $t\geq s\geq\tau$ and
the function $[\tau,\infty)\ni t\mapsto T(t,\tau)\eta$ is continuous for each $(\tau,\eta)\in\R\times X.$
Moreover, $T(t,\tau)0=0$ for $t\geq\tau$.

 Assuming the existence of an exponential splitting for the linear process $\{L(t,\tau)\colon t\geq \tau\}$, we show that this property is, in a certain sense, preserved under nonlinear perturbations. This behavior leads to the existence of an invariant manifold, which we now recall.
 
\begin{defn}\label{defn:INERTIAL}
	A family $\{\calm(t)\colon t\in\R\}\subset X$ is called an invariant manifold for the evolution process $\{T(t,\tau)\colon t\geq\tau\}$ if
	\begin{itemize}
		\item[(i)] $\{\calm(t)\colon t\in\R\}$ is invariant under the process, i.e., 
		$$T(t,\tau)\calm(\tau)=\calm(t)\text{ for all }t,\tau\in\R\text{ such that }t\geq\tau;$$
		\item[(ii)] $\{\mathcal{M}(t)\colon t\in \mathbb{R}\}$ is \textit{forward and pullback exponentially dominated}, i.e.,
		there is $\omega\in \mathbb{R}$ such that
given a bounded set $U\subset X$, there exist $t_*\geq 0$, $K>0$ such that 
$$\dist(T(t,\tau) U, \mathcal{M}(t))\leq Ke^{-\omega(t-\tau)},
$$ 
for all $t,\tau \in \R$ with  $t-\tau\geq t_*$;
		\item[(iii)] $\calm(t)$ is a Lipschitz graph for each $t\in\R$. 
	\end{itemize}
If additionally $\omega>0$, then $\{\mathcal{M}(t)\colon t\in \mathbb{R}\}\subset X$ is \textit{exponentially attracting} (pullback and forward) and it is called an {\emph{inertial manifold}.} 
\end{defn}

When studying invariant objects, complete trajectories that capture their behavior naturally arise. In \cite{CLR}, these special solutions are referred to as global solutions.

\begin{defn}
We say that $\xi\colon\calt\to X$ is a solution of the evolution process $\{T(t,\tau)\colon (t,\tau)\in J\}$ (or of \eqref{e:VCF}, for short) on an interval $\calt\subset\R$ if 
$T(t,\tau)\xi(\tau)=\xi(t)\ \text{for}\ t\geq \tau,\ t,\tau\in\calt.$
We call it a~global solution (or complete trajectory) if $\calt=\R$.
\end{defn}
Note that by the assumed continuity of the evolution process its solutions are continuous functions with values in $X$.

We assume the gap condition given by the following definition.
\begin{defn}\label{def:gap}
	We say that the constants $\gamma>\rho$, and $L_1,L_2>0$ satisfy the gap condition with the admissible norm $\Gamma(\cdot,\cdot)$ on $\R^2$ if there exists $\sigma\in (\rho,\gamma)$ such that 
		\begin{equation}\label{e:GAPGENERAL00}
	\Gamma\left(\frac{L_1}{\sigma-\rho},\frac{L_2}{\gamma-\sigma}\right)<1.
	\end{equation}
\end{defn}

\begin{cor}
Assume that
\begin{equation}\label{e:GAPGENERAL0}
\gamma-\rho > \Gamma(1,1)(L_1+L_2),
\end{equation}
then the gap condition given in Definition \ref{def:gap} is satisfied with any $\sigma\in(\rho+\Gamma(1,1)L_1,\gamma-\Gamma(1,1)L_2)$.
\end{cor}

\begin{proof}
We can choose $\sigma$ such that 
$$\gamma-L_2\Gamma(1,1)>\sigma>\rho+L_1\Gamma(1,1),$$
which is equivalent to the fact that $\frac{L_1}{\sigma-\rho}<\frac{1}{\Gamma(1,1)}$ and $\frac{L_2}{\gamma-\sigma}<\frac{1}{\Gamma(1,1)}$. By the monotonicity of $\Gamma$, we have
$$\Gamma\left(\frac{L_1}{\sigma-\rho},\frac{L_2}{\gamma-\sigma}\right)<1,$$
which ends the proof. 
\end{proof}

\begin{rem}\label{rem:SHARPNESS}
Note that if \eqref{e:GLOBALLIP} holds, and we have the condition 
\begin{equation*}
	\gamma-\rho>4M\ell,
\end{equation*}
then, regardless of the chosen norm $\Gamma(\cdot)=\norm{\cdot}_p$ with $1\leq p\leq \infty$, the gap condition in Definition~\ref{def:gap} holds with any $\sigma\in(\rho+2M\ell,\gamma-2M\ell)$. This improves the gap condition of \cite{CLMO-S25}. 

If we choose the norm $\|(a_1,a_2)\|_\infty$ in the role of $\Gamma$, then the inequality \eqref{e:GAPGENERAL00} holds
\emph{if and only if} $\sigma\in(\rho-L_1,\gamma+L_2)$. Thus, in this case, the gap condition in Definition~\ref{def:gap} is \emph{equivalent} to the optimal condition of \cite{LL99}, which is exactly~\eqref{e:GAPGENERAL0}, namely
$$\gamma-\rho>L_1+L_2.$$ 
If additionally $L_1 = L_2 =L$, this reduces to the sharp gap condition 
\begin{equation}\label{e:sharp_gap}
\gamma-\rho>2L.
\end{equation}
Sharpness of condition \eqref{e:sharp_gap} was first obtained in \cite{Miklavcic, Romanov}, see also \cite{CR96,LL99,Zelik}, and it can be seen in the example of \cite{CR96}, namely in the system
$$
\begin{pmatrix}
	x \\ y 
\end{pmatrix}' = \begin{pmatrix}
1 & 0 \\ 0 & -1 
\end{pmatrix} 
\begin{pmatrix}
	x \\ y 
\end{pmatrix} + \varepsilon\begin{pmatrix}
- y \\  x 
\end{pmatrix},
$$
where the term with $\varepsilon$ is treated as nonlinearity. Then $\gamma=1$, $\rho=-1$, and $L=L_1=L_2=|\varepsilon|$. The gap condition \eqref{e:sharp_gap} signifies that $|\varepsilon| < 1$. The constant $1$ in the right-hand side of the above formula cannot be increased, because if we take $\varepsilon$ greater than one, we do not have the invariant set which is a graph over the $x$ variable. 

Note that in this norm, the Lipschitz conditions \eqref{e:L1}--\eqref{e:L2} on the nonlinearity $f$ with $L_1=L_2=L$ can be written as 
$$\|f(\tau,u)-f(\tau,v)\|_\tau \leq L\|u-v\|_\tau\ \ \text{for}\ \ \tau\in \R.$$

In the case when the norm $\|(a_1,a_2)\|_p$ is taken in the role of $\Gamma$, the condition \eqref{e:GAPGENERAL00} is equivalent to say that
$$
\gamma-\rho > \left(L_1^{\frac{p}{p+1}}+L_2^{\frac{p}{p+1}}\right)^\frac{p+1}{p}.
$$
In particular, if $L_1=L_2=L$, this takes the (non-sharp) form
$$\gamma-\rho > 2^{1+\frac{1}{p}}L.$$
In particular, in the case of the norm $\|(a_1,a_2)\|_1$ we obtain the condition 
$\gamma-\rho>4L,$ and in the case of the norm $\|(a_1,a_2)\|_2$ we get $\gamma-\rho>2\sqrt{2}L.$ 
\end{rem}

\subsection{Gronwall lemmas.} We conclude this section by recalling two integral versions of the Gronwall lemma that will be both used in the sequel.
\begin{lem}[Gronwall]
Let $-\infty\leq a<b<\infty$ and let $y\colon(a,b]\to\R$ be a locally Lebesgue integrable function. If  a non-increasing function $\alpha\colon(a,b]\to\R$ and $\beta>0$ satisfy the inequality
$$y(t)\leq\alpha(t)+\beta\int_{t}^{b}y(s)ds,\ t\in(a,b],$$
then
\begin{equation}\label{e:GRONWALL1}
y(t)\leq\alpha(t)e^{\beta(b-t)},\ t\in(a,b].
\end{equation} 
\end{lem}

\begin{proof}
We multiply by $\beta e^{\beta(t-b)}$ and get
$$\beta y(t)e^{\beta(t-b)}\leq\alpha(t)\beta e^{\beta(t-b)}+\beta^2 e^{\beta(t-b)}\int_{t}^{b}y(s)ds,\ t\in(a,b],$$
so
$$\frac{d}{dt}\Bigl(-\beta e^{\beta(t-b)}\int_{t}^{b}y(s)ds \Bigr)\leq\alpha(t)\beta e^{\beta(t-b)}.$$
Integrating on the interval $[t,b]$ we obtain
$$\beta e^{\beta(t-b)}\int_{t}^{b}y(s)ds \leq \beta\int_{t}^{b}\alpha(s)e^{\beta(s-b)}ds\leq \beta\alpha(t)\int_{t}^{b}e^{\beta(s-b)}ds=\alpha(t)(1-e^{\beta(t-b)}),$$
which leads to \eqref{e:GRONWALL1}.
\end{proof}

\begin{lem}[Gronwall]\label{lem:GRONWALL2}
Let $-\infty\leq a<b<\infty$ and let $y\colon(a,b]\to\R$ and $\alpha\colon(a,b]\to\R$ be locally Lebesgue integrable functions. If for some $\beta>0$ they satisfy the inequality
$$y(t)\leq\int_{t}^{b}(\alpha(s)+\beta y(s))ds,\ t\in(a,b],$$
then
$$y(t)\leq\int_{t}^{b}\alpha(s)e^{\beta(s-t)}ds,\ t\in(a,b].$$
\end{lem}

\begin{proof}
Defining
$$z(s)=\int_{s}^{b}(\alpha(r)+\beta y(r))dr,\ s\in(a,b],$$
we get
$$\frac{d}{ds}\bigl[-z(s)e^{\beta s}\bigr]=(-z'(s)-\beta z(s))e^{\beta s}\leq\alpha(s)e^{\beta s},\ s\in(a,b].$$
Thus
$$z(t)e^{\beta t}\leq\int_{t}^{b}\alpha(s)e^{\beta s}ds,\ s\in(a,b],$$
which yields the claim.
\end{proof}

\section{Results on existence of Lipschitz invariant and inertial manifolds}\label{sec:LIP}

\subsection{Main result.} We first formulate the main theorem of this paper, which will be proved in the following part of this section.

\begin{thm}\label{thm:MAIN1}
	Assume that the linear abstract Cauchy problem \eqref{e:LACP} generates a linear process $\{L(t,\tau)\colon t\geq\tau\}$ on a Banach space $X$, which has the exponential splitting (Definition \ref{def:splitting}) with projections $\{Q(\tau)\colon \tau\in\R\}$ and constants $\gamma,\rho\in\R$ and $M\geq 1$.
	
	Introducing the spaces $N(\tau)=Q(\tau)X$ and $S(\tau)=(I-Q(\tau))X$ for $\tau\in\R$, endowed with the norms $|\cdot|_{N(\tau)}|$ and $|\cdot|_{S(\tau)}$ given by \eqref{e:NORMNT} and \eqref{e:NORMST}, respectively,
	and equipping $X$ with the equivalent norms $\norm{\cdot}_{\tau}$ given in \eqref{e:MOVINGNORM} for a chosen admissible norm $\Gamma$ on $\R^2$ (Definition~\ref{defn:ADMISSIBLENORM}), we further assume that a continuous function $f\colon\R\times X\to X$ satisfies \eqref{e:ZEROCONDITION} and the uniform global Lipschitz conditions  
	\eqref{e:L1} and \eqref{e:L2} with constants $L_1,L_2>0$.
	
	Consider the perturbed abstract Cauchy problem \eqref{e:PERTURBED} and assume that it generates an evolution process $\{T(t,\tau)\colon t\geq\tau\}$ in $X$, which satisfies the variation of constants formula \eqref{e:VCF}, that is,
	\begin{equation*}
		T(t,\tau)\eta=L(t,\tau)\eta+\int_{\tau}^{t}L(t,s)f(s,T(s,\tau)\eta)ds\ \ \text{for}\ \  t\geq\tau\ \ \text{and}\ \ \eta\in X.
	\end{equation*}
	Finally, assume that the gap condition \eqref{e:GAPGENERAL00} given in Definition \ref{def:gap} holds, that is,
	$$\Gamma\left(\frac{L_1}{\sigma-\rho},\frac{L_2}{\gamma-\sigma}\right)<1\ \ \text{for some}\ \ \sigma\in (\rho,\gamma).$$ 
    
    Then there exists an invariant manifold $\{\mathcal{M}(t)\colon t\in \mathbb{R}\}$ for $\{T(t,\tau)\colon t\geq \tau\}$ given by 
\begin{equation}\label{eq-characterization-inertail}
\begin{split}
\mathcal{M}(t)=\Big\{\eta\in X\colon &\text{there exists a global solution}\ z\colon\R\to X\ \text{of}\ \eqref{e:VCF} \\
&\text{such that } z(t)=\eta, \ \sup_{r\leq t }\{e^{\sigma r}\|z(r)\|_r\}<+\infty\Big\},
\end{split} 
\end{equation}

satisfying the following properties:
\begin{enumerate}
    \item[(I)] $\{\mathcal{M}(t)\colon t\in \R\}$ is a Lipschitz graph:
    $$\mathcal{M}(t)=P_\Sigma(t)X:=\{Q(t)\eta+\Sigma(t,Q(t)\eta)\colon \eta\in X\},$$ 
    where {$P_\Sigma(t)\eta:=Q(t)\eta+\Sigma(t,\eta)$, $(t,\eta)\in\mathbb{R}\times X$,} for a function
$\Sigma\colon\mathbb{R}\times X \to X$, (defined later in \eqref{e:DEFNSIGMA}) such that
$$\Sigma(\tau,\eta)=\Sigma(\tau,Q(\tau)\eta)=(I-Q(\tau))\Sigma(\tau,\eta)\ \text{for}\ \tau\in\R\ \text{and}\ \eta\in X,$$
	 $\Sigma(\tau,0)=0$ for $\tau\in\R$, and
	\begin{equation*}
		|\Sigma(\tau,\eta)-\Sigma(\tau,\tilde\eta)|_{S(\tau)}\leq\kappa_\Sigma|Q(\tau)(\eta-\tilde\eta)|_{N(\tau)}\ \text{for}\  \eta,\tilde\eta\in X\ \text{and}\ \tau\in\R,
	\end{equation*}
	with the Lipschitz constant $0<\kappa_\Sigma<\kappa = \frac{L_2}{L_1}\frac{\sigma-\rho}{\gamma-\sigma}$ (see Corollary~\ref{cor:betterlip})
	 satisfying
\begin{equation*}
\gamma-\rho=L_1\Gamma(1,\kappa_\Sigma)+L_2\Gamma\left(\frac{1}{\kappa_\Sigma},1\right),
\end{equation*} 	
so that the value of $\kappa_\Sigma$ depends only on $\gamma, \rho, L_1, L_2$ and the chosen admissible norm $\Gamma$.
\item[(II)] $\{\mathcal{M}(t)\colon t\in \R\}$ satisfies the following refined controlled growth
{
    \begin{equation*}
    \|z(t)\|_t \leq \frac{\Gamma(1,\kappa_\Sigma)}{\Gamma(1,0)}e^{-(\rho+L_1 \Gamma(1,\kappa_\Sigma))(t-\tau)}\|\eta\|_\tau\  \text{for}\ t\leq\tau,
    \end{equation*}
    where $z\colon\R\to X$ is a unique global solution of \eqref{e:VCF} through $z(\tau)=\eta\in\mathcal{M}(\tau)$,}
    \item[(III)] $\{\mathcal{M}(t)\colon t\in \R\}$ satisfies the following  property
$$|T(t,\tau)\eta-P_\Sigma(t)T(t,\tau)\eta|_{S(t)}\leq |\eta-P_\Sigma(\tau)\eta|_{S(\tau)}e^{-\omega(t-\tau)}\ \text{for}\  t\geq\tau,\ \eta\in X,$$
with
\begin{equation*}
\omega=\gamma-\frac{(\gamma-\rho)L_2\Gamma(0,1)}{\gamma-\rho-L_1\Gamma(1,\kappa_\Sigma)}\in(\rho,\gamma),
\end{equation*}
      \item[(IV)] $\{\mathcal{M}(t)\colon t\in \R\}$  exponentially controls the evolution of any bounded set $G\subset X$ in the following sense: for each bounded subset $G\subset X$ there is $C_G>0$ such that 
\begin{equation*}
\dist(T(t,\tau)G,\mathcal{M}(t))\leq C_Ge^{-\omega(t-\tau)}\ \ \text{for}\  \ t\geq \tau,
\end{equation*}
where $\dist(A,B)=\sup_{a\in A}\inf_{b\in B}{\|a-b\|}$ is the Hausdorff semi-distance between $A,B\subset X$.
\end{enumerate}
If additionally $\omega>0$, then $\{\calm(t)\colon t\in\R\}$ is a forward and pullback exponentially attracting invariant manifold, i.e., an inertial manifold for the evolution process $\{T(t,\tau)\colon t\geq\tau\}$.

\end{thm}

We now briefly recall the definition of a pullback attractor, see \cite{CLR} for {more details on} 
this object. 
    A \textit{pullback attractor} 
    {$\{\cala(t)\colon t\in\R\}$ for $\{T(t,\tau)\colon t\geq\tau\}$} is a family of compact subsets of $X$ such that $\{\cala(t)\colon t\in\R\}$ is invariant, pullback attracts bounded subsets of $X$, i.e., for each bounded subset $B$ of $X$, and $t
\in \R$, 
\begin{equation*}
    \lim_{
    \tau\to -\infty}\dist(T(t,\tau)B,\mathcal{A}(t))=0,
\end{equation*}
    and it is the minimal family of closed subsets which pullback attracts bounded subsets.
\begin{rem} 
    If $\{\calm(t)\colon t\in\R\}$ is an inertial manifold (if $\omega>0$ in Theorem \ref{thm:MAIN1}) and there exists pullback attractor $\{\cala(t)\colon t\in\R\}$ for $\{T(t,\tau)\colon t\geq\tau\}$, then minimality implies that $\cala(t)\subset \calm(t)$, for all $t\in\R$.
\end{rem}

\subsection{Fixed point argument.}\label{subsec:ZET} 

 Let $z\colon(-\infty,\tau]\to X$ be a solution of \eqref{e:VCF}. 
 Then after representing $Q(t)z(t) = q(t)$ and $(I-Q(t))z(t)=p(t)$, projecting \eqref{e:VCF} on $N(\tau)$ and {$S(t)$}, the following variation of constants formulas hold
\begin{align*}
    & q(\tau)=L(\tau,t)q(t)+\int_{t}^{\tau}L(\tau,s)Q(s)f(s,z(s))ds\ \ \text{for}\ \ {-\infty<t\leq\tau},\\
    & p(t)=L(t,r)p(r)+\int_{r}^{t}L(t,s)(I-Q(s))f(s,z(s))ds\ \ \text{for}\ \  {-\infty<r\leq t\leq\tau.}
\end{align*}
In the first of these equations we use the invertibility of $L(\tau,t)$, which yields
$$q(t) = L(t,\tau)q(\tau)-\int_{t}^{\tau}L(t,s)Q(s)f(s,z(s))ds\ \ \text{for}\ \  -\infty<t\leq\tau.$$

To deal with  the second equation note that for $\sigma<\gamma$ we have
$$
|L(t,r)p(r)|_{S(t)} \leq e^{-\gamma(t-r)}|p(r)|_{S(r)} \leq \frac{e^{(\gamma-\sigma)r}}{\Gamma(0,1)}e^{-\gamma t}e^{\sigma\tau}\sup_{r\leq \tau}\left\{e^{\sigma(r-\tau)}\|z(r)\|_r\right\},$$
which tends to zero as $r\to -\infty$ provided the last supremum is finite. Then this means that
\begin{equation}\label{eq:fixed_point}
z(t) = L(t,\tau)q(\tau)-\int_{t}^{\tau}L(t,s)Q(s)f(s,z(s))ds  + \int_{-\infty}^{t}L(t,s)(I-Q(s))f(s,z(s))ds,\ {t\leq\tau.}
\end{equation}
Above equation motivates the following fixed point problem.

Given $\tau\in\R$ we consider the space 
$$E_{\sigma,\tau}=\left\{z\in C((-\infty,\tau];X)\colon \norm{z}_{E_{\sigma,\tau}}=\sup_{t\leq\tau}e^{\sigma(t-\tau)}\norm{z(t)}_{t}<\infty\right\},$$
which is a Banach space with the given norm. Having fixed $\tau\in\R$, $\eta\in X$, we consider the mapping $H_{\tau,\eta}\colon E_{\sigma,\tau}\to C((-\infty,\tau];X)$ given by the formula
$$H_{\tau,\eta}z(t)=L(t,\tau)Q(\tau)\eta-\int_{t}^{\tau}L(t,s)Q(s)f(s,z(s))ds+\int_{-\infty}^{t}L(t,s)(I-Q(s))f(s,z(s))ds,$$
for  $-\infty<t\leq\tau$.
\begin{lem}\label{lem:fixed_point}
	Assume the gap condition \eqref{e:GAPGENERAL00}. The mapping $H_{\tau,\eta}$ is a contraction leading from $E_{\sigma,\tau}$ to itself, and hence it has a unique fixed point $z_{\tau,\eta}\in E_{\sigma,\tau}$. 
\end{lem}
\begin{proof}
We have
\begin{align*}&\norm{H_{\tau,\eta}z(t)}_{t}\\
	&=\Gamma\left(\left|L(t,\tau)Q(\tau)\eta-\int_{t}^{\tau}L(t,s)Q(s)f(s,z(s))ds\right|_{N(t)},\left|\int_{-\infty}^{t}L(t,s)(I-Q(s))f(s,z(s))ds\right|_{S(t)}\right).\end{align*}
We estimate both terms in the above norm as follows 
$$\left|L(t,\tau)Q(\tau)\eta-\int_{t}^{\tau}L(t,s)Q(s)f(s,z(s))ds\right|_{N(t)}\leq e^{\rho(\tau-t)}|Q(\tau)\eta|_{N(\tau)}+L_1\int_{t}^{\tau}e^{\rho(s-t)}\norm{z(s)}_{s}ds,$$
and
$$\left|\int_{-\infty}^{t}L(t,s)(I-Q(s))f(s,z(s))ds\right|_{S(t)}\leq L_2\int_{-\infty}^{t}e^{\gamma(s-t)}\norm{z(s)}_{s}ds.$$
Thus
\begin{align*}
&e^{\sigma(t-\tau)}\norm{H_{\tau,\eta}z(t)}_{t}\\
&\ \ \leq \Gamma\left(e^{(\sigma-\rho)(t-\tau)}|Q(\tau)\eta|_{N(\tau)}+L_1\norm{z}_{E_{\sigma,\tau}}\int_{t}^{\tau}e^{(\sigma-\rho)(t-s)}ds,L_2\norm{z}_{E_{\sigma,\tau}}\int_{-\infty}^{t}e^{(\gamma-\sigma)(s-t)}ds\right)\\
&\ \ =\Gamma\left(e^{(\sigma-\rho)(t-\tau)}|Q(\tau)\eta|_{N(\tau)}+L_1\norm{z}_{E_{\sigma,\tau}}\frac{1-e^{(\sigma-\rho)(t-\tau)}}{\sigma-\rho},L_2\norm{z}_{E_{\sigma,\tau}}\frac{1}{\gamma-\sigma}\right)\\
&\ \ \leq\Gamma\left(|Q(\tau)\eta|_{N(\tau)}+L_1\norm{z}_{E_{\sigma,\tau}}\frac{1}{\sigma-\rho},L_2\norm{z}_{E_{\sigma,\tau}}\frac{1}{\gamma-\sigma}\right).
\end{align*}
Therefore, since $\gamma>\sigma>\rho$, we get
\begin{equation*}
\norm{H_{\tau,\eta}z}_{E_{\sigma,\tau}}\leq \Gamma\left(|Q(\tau)\eta|_{N(\tau)}+L_1\norm{z}_{E_{\sigma,\tau}}\frac{1}{\sigma-\rho},L_2\norm{z}_{E_{\sigma,\tau}}\frac{1}{\gamma-\sigma}\right)<\infty,
\end{equation*}
and $H_{\tau,\eta}$ leads from $E_{\sigma,\tau}$ to itself. 

To prove that it is a contraction, estimating in a similar way, we obtain
$$\norm{H_{\tau,\eta}z_1-H_{\tau,\eta}z_2}_{E_{\sigma,\tau}}\leq\norm{z_1-z_2}_{E_{\sigma,\tau}}\Gamma\left(\frac{L_1}{\sigma-\rho},\frac{L_2}{\gamma-\sigma}\right)\ \ \textrm{for all}\ \ z_1,z_2\in E_{\sigma,\tau},$$
which ends the proof by \eqref{e:GAPGENERAL00}.
\end{proof}

Thus there exists a unique function $z_{\tau,\eta}\in E_{\sigma,\tau}$ such that
$$z_{\tau,\eta}(t)=L(t,\tau)Q(\tau)\eta-\int_{t}^{\tau}L(t,s)Q(s)f(s,z_{\tau,\eta}(s))ds+\int_{-\infty}^{t}L(t,s)(I-Q(s))f(s,z_{\tau,\eta}(s))ds,$$
It is clear that $z_{\tau,\eta}$ does not depend on $(I-Q(\tau))\eta$ and hence $z_{\tau,\eta}=z_{\tau,Q(\tau)\eta}$ for $\tau\in\R$ and $\eta\in X$. Moreover $z_{\tau,0}=0$ for every $\tau\in\R$.

\subsection{Definition of the manifold and its invariance.} 
 In this subsection we will use the notation 
$$q(t)=Q(t)z(t)\quad\text{and}\quad p(t)=(I-Q(t))z(t)\ \ \text{for}\ \ z\in C((-\infty,\tau];X).$$
\begin{lem}\label{lem:lemma_solution}
    Let $\eta\in X$ and $\tau\in \R$ be given. The function $z$ is a fixed point of $H_{\tau,\eta}$ in $E_{\sigma,\tau}$ if and only if $z\in E_{\sigma,\tau}$ is a solution of \eqref{e:VCF} and $Q(\tau)z(\tau) = Q(\tau)\eta$. 
\end{lem}
\begin{proof}
    The computation above Lemma \ref{lem:fixed_point} shows that a solution of \eqref{e:VCF} defined on the interval $(-\infty,\tau]$ which belongs to $E_{\sigma,\tau}$ and satisfies $Q(\tau)z(\tau) = Q(\tau)\eta$ is a fixed point of $H_{\tau,\eta}$. For an opposite implication observe that for $-\infty<t\leq\theta\leq\tau$ we have 
\begin{equation*}
\begin{split} 
q(t)&=L(t,\tau)Q(\tau)\eta-\int_{t}^{\tau}L(t,s)Q(s)f(s,z(s))ds\\
&=L(t,\theta)\left(L(\theta,\tau)Q(\tau)\eta-\int_{\theta}^{\tau}L(\theta,s)Q(s)f(s,z(s))ds\right)-\int_{t}^{\theta}L(t,s)Q(s)f(s,z(s))ds\\
& =L(t,\theta)q(\theta)-\int_{t}^{\theta}L(t,s)Q(s)f(s,z(s))ds.
\end{split}
\end{equation*}
Inverting $L(t,\theta)$ yields
$$
q(\theta) = L(\theta,t)q(t) + \int_t^\theta L(\theta,s)Q(s)f(s,z(s))\, ds.
$$
On the other hand, for $-\infty<t\leq \theta\leq\tau$ {we have}
\begin{equation*}
\begin{split}
p(\theta)&=\int_{-\infty}^{\theta}L(\theta,s)(I-Q(s))f(s,z(s))ds\\
&=L(\theta,t)\int_{-\infty}^{t}L(t,s)(I-Q(s))f(s,z(s))ds+\int_{t}^{\theta}L(\theta,s)(I-Q(s))f(s,z(s))ds\\
&=L(\theta,t)p(t)+\int_{t}^{\theta}L(\theta,s)(I-Q(s))f(s,z(s))ds,
\end{split}
\end{equation*}
and the proof is complete. 
\end{proof}
Note that the uniqueness of the fixed point of $H_{\tau,\eta}$ and the above lemma imply that for a given $\eta\in X$ and $\tau\in \R$ there exists a unique function $z\in E_{\sigma,\tau}$ which is a solution of \eqref{e:VCF}  such that $Q(\tau)z=Q(\tau)\eta$ and at the same time the fixed point $z_{\tau,\eta}$ of $H_{\tau,\eta}$. We  are in a position to define the function $\Sigma\colon\R\times X\to X$, whose graph is a candidate for the manifold that we are constructing,
\begin{equation}\label{e:DEFNSIGMA}
\Sigma(\tau,\eta)=(I-Q(\tau))z_{\tau,\eta}(\tau)=\int_{-\infty}^{\tau}L(\tau,s)(I-Q(s))f(s,z_{\tau,\eta}(s))ds\ \ \text{for}\ \ \tau\in\R\ \ \text{and}\ \ \eta\in X.
\end{equation}
From its construction, we know that $\Sigma(\tau,\eta)=\Sigma(\tau,Q(\tau)\eta)=(I-Q(\tau))\Sigma(\tau,\eta)$ for $\tau\in\R$ and $\eta\in X$, and $\Sigma(\tau,0)=0$ for $\tau\in\R$.
Now, we define  
$\mathcal{M}(\tau)$ as the graph of $\Sigma(\tau,
\cdot):Q(\tau)X\to (I-Q(\tau))X$, that is, 
\begin{equation}\label{e:MANIFOLD}
\calm(\tau):=\{Q(\tau)\eta+\Sigma(\tau,Q(\tau)\eta)\colon \eta\in X\}=\{z_{\tau,\eta}(\tau)\colon\eta\in X\}.
\end{equation}
In other words, 
$\{\mathcal{M}(\tau)\colon \tau\in \R\}$
is a family of subsets of $X$ given by
the image of the nonlinear projections, 
\begin{equation}\label{eq-Def_P_Sigma}
    P_\Sigma(\tau):=Q(\tau)+\Sigma(\tau,\cdot) \text{ for } \tau
\in \R.
\end{equation}
This family satisfies the following result.
\begin{lem}\label{lem:characterization}
    Fix $\tau\in \R$. We have $\eta\in \mathcal{M}(\tau)$ if and only if there exists a solution  $z\colon(-\infty,\tau]\to X$ of \eqref{e:VCF} such that $z(\tau) = \eta$ and $z\in E_{\sigma,\tau}$. In consequence, this means that  \eqref{eq-characterization-inertail} holds.
\end{lem}

\begin{proof}
Note that $\eta\in \mathcal{M}(\tau)$ if and only if $\Sigma(\tau,\eta)=(I-Q(\tau))\eta$. For $\eta\in \mathcal{M}(\tau)$ we consider the unique fixed point $z_{\tau,\eta}$ of $H_{\tau,\eta}$. By Lemma \ref{lem:lemma_solution} this function is a solution of \eqref{e:VCF}, belongs for $E_{\sigma,\tau}$, and $Q(\tau)z_{\tau,\eta}(\tau) = Q(\tau)\eta$.  Moreover, by definition of $\Sigma$, we have $(I-Q(\tau))z_{\tau,\eta}(\tau) = \Sigma(\tau,\eta)$ and hence $z_{\tau,\eta}(\tau) =\eta$. 

Now assume that $z\colon(-\infty,\tau]\to X$ is a solution of \eqref{e:VCF} such that $z(\tau)=\eta$ and $z\in E_{\sigma,\tau}$. Then, by Lemma \ref{lem:lemma_solution} the function $z$ is a unique fixed point of 
$H_{\tau,\eta}$, and it must be equal to $z_{\tau,\eta}$. Therefore, by definition, $\Sigma(\tau,\eta) = (I-Q(\tau))z(\tau)$. This means that $\eta=z(\tau) = Q(\tau)\eta + \Sigma(\tau,\eta)\in \calm(\tau).$

Extending $z(t)=T(t,\tau)\eta$ for $t\geq \tau$, we get a global solution of \eqref{e:VCF} so that \eqref{eq-characterization-inertail} holds.
\end{proof}

The above result allows us to easily deduce the invariance of constructed family of manifolds $\{\calm(\tau)\colon \tau\in \R \}$. 

\begin{lem}\label{lem:invariance}
The family $\{\calm(\tau)\colon \tau\in \R \}$ is invariant, that is, $T(t,\tau)\calm(\tau) = \calm(t)$ for every $(t,\tau)\in J$. 
\end{lem}
\begin{proof}
    Take $\tau\in \R$, $\eta \in \calm(\tau)$, and  $t>\tau$. By Lemma \ref{lem:characterization} there exists a function $z\in E_{\sigma,\tau}$, solution of \eqref{e:VCF} on interval $(-\infty,\tau]$ such that $z(\tau) = \eta$. Define
    $$
    v(s) = \begin{cases}
        z(s)\ \ & \text{for}\ \ s\in (-\infty,\tau],\\
        T(s,\tau) \eta\ \ &\textrm{for}\ \ s\in (\tau,t].
    \end{cases}
    $$
    This is a solution of \eqref{e:VCF} on $(-\infty,t]$ belonging to $E_{\sigma,t}$. Hence $v(t)=T(t,\tau)\eta\in \calm(t)$ {so} we have the positive invariance.

    To deduce the negative invariance take $t\in \R$, $\eta \in \calm(t)$, and $\tau<t$. By Lemma \ref{lem:characterization} there exists $z\in E_{\sigma,t}$ which solves  \eqref{e:VCF} on the interval $(-\infty,t]$ such that $z(t)=\eta$. Consider $z|_{(-\infty,\tau]}$. This function belongs to $E_{\sigma,\tau}$ and solves 
\eqref{e:VCF} on $(-\infty,\tau]$. Hence $\xi = z|_{(-\infty,\tau]}(\tau) \in \calm(\tau)$. As $T(t,\tau)\xi = \eta$ it follows that $\calm(t)\subset T(t,\tau)\calm(\tau)$  and we have the negative invariance.
\end{proof}

\begin{rem}
Note that the results of this section imply that the invariant manifold is defined uniquely for a given family of projections $\{Q(t)\colon t\in \R\}$ for which we have the exponential splitting such that the gap condition \eqref{e:GAPGENERAL00} holds. Indeed, choosing the two exponents $\sigma_1<\sigma_2$ for which the gap condition holds, we denote the two constructed manifolds by $\calm_{\sigma_1}$ and $\calm_{\sigma_2}$ with the corresponding functions $\Sigma_{\sigma_1}$ and $\Sigma_{\sigma_2}$. Pick $t\in \R$. By \eqref{eq-characterization-inertail} we must have $\calm_{\sigma_1}(t)\subset \calm_{\sigma_2}(t)$. Suppose that $\eta \in \calm_{\sigma_2}(t) \setminus  \calm_{\sigma_1}(t)$. Consider $Q(t)\eta+\Sigma_{\sigma_1}(t,Q(t)\eta)$. This point belongs to $\calm_{\sigma_1}(t)$ and hence also to  $\calm_{\sigma_2}(t)$. But $\eta$ is a unique point in $\calm_{\sigma_2}(t)$ whose projection $Q(t)$ is equal to $Q(t)\eta$ and hence $\eta = Q(t)\eta+\Sigma_{\sigma_1}(t,Q(t)\eta) \in \calm_{\sigma_1}(t)$, a contradiction.
\end{rem}

\color{black}

\subsection{Cone condition and Lipschitzness of \texorpdfstring{$\Sigma$}.}
This section is devoted to the verification of the Lipschitzness of the function $\Sigma(t,\cdot)$ that is used to define the manifold $\{\calm(t)\colon t\in\R\}$. This will be done by means of the \textit{cone condition}. \color{black}  Assume the gap condition \eqref{e:GAPGENERAL00} given in Definition~\ref{def:gap} and let 
\begin{equation}\label{e:DEFNKAPPA}
\kappa = \frac{L_2}{L_1}\frac{\sigma-\rho}{\gamma-\sigma}. 
\end{equation}
We consider two solutions of \eqref{e:VCF}, which we call $z, \tilde z$, defined on an interval $\calt\subset\R$, which can be bounded or unbounded. We denote $$q(t)=Q(t)z(t),\ p(t)=(I-Q(t))z(t),\ \tilde q(t)=Q(t)\tilde z(t),\ \tilde p(t)=(I-Q(t))\tilde z(t)\ \ \text{for}\ \ t\in\calt.$$
Moreover, we set
$$u=q-\tilde q,\ v=p-\tilde{p}.$$
Then $t\mapsto |u(t)|_{N(t)}$ and $t\mapsto |v(t)|_{S(t)}$ are continuous functions on $\calt$. 
Let 
\begin{equation}\label{e:DEFNZETA}
\zeta(t):=|v(t)|_{S(t)}-\kappa |u(t)|_{N(t)}\ \text{for}\ t\in\calt.
\end{equation}
We will study the behavior of the function $\zeta\colon \calt\to \R$. 

\begin{lem}\label{lem:POS}
If the gap condition \eqref{e:GAPGENERAL00} holds then for any $r<\theta$, $r,\theta\in\calt$ if only $u(\theta)\neq 0$  or $v(r)\neq 0$ then 
$$\zeta(t) \geq 0\ \ \text{ for }\ \ t\in [r,\theta]\ \ \text{ implies }\ \ \zeta(r)>0.$$ 
\end{lem}

\begin{proof}
Suppose contrary to the claim that $\zeta(r)=0$, that is, $|v(r)|_{S(r)}=\kappa |u(r)|_{N(r)}$. Subtracting the equations for $z$ and $\tilde z$ and applying the projection $Q(\theta)$,  we obtain
$$u(t)=L(t,\theta)u(\theta)-\int_{t}^{\theta}L(t,s)Q(s)[f(s,z(s))-f(s,\tilde z(s))]ds\ \ \textrm{for}\ \ t\leq\theta,\ t,\theta\in\calt.$$
Thus
$$|u(t)|_{N(t)}\leq e^{\rho(\theta-t)}|u(\theta)|_{N(\theta)}+L_1\int_{t}^{\theta}e^{\rho(s-t)}\Gamma(|u(s)|_{N(s)},|v(s)|_{S(s)})ds .$$
By assumption we obtain
\begin{equation}\label{e:ESTU1}
|u(t)|_{N(t)}\leq  e^{\rho(\theta-t)}|u(\theta)|_{N(\theta)}+L_1\Gamma\left(\frac{1}{\kappa},1\right)\int_{t}^{\theta}e^{\rho(s-t)}|v(s)|_{S(s)}ds\ \ \text{for}\ \ t\in [r,\theta].
\end{equation}
Applying the projection $I-Q(t)$ to the difference of equations for $z$ and $\tilde z$ we get
$$v(t)=L(t,r)v(r)+\int_{r}^{t}L(t,s)(I-Q(s))[f(s,z(s))-f(s,\tilde z(s))]ds\ \ \text{for}\ \ t\in [r,\theta],$$
which by assumption leads to
\begin{equation}\label{e:secondbound}
	|v(t)|_{S(t)}\leq e^{\gamma(r-t)}|v(r)|_{S(r)}+L_2\Gamma\left(\frac{1}{\kappa},1\right)\int_{r}^{t}e^{\gamma(s-t)}|v(s)|_{S(s)}ds\ \ \text{for}\ \ t\in [r,\theta].
	\end{equation}
By the Gronwall inequality we obtain from the  above  bound
\begin{equation}\label{e:ESTV1}
|v(t)|_{S(t)}\leq |v(r)|_{S(r)}e^{\left(\gamma-L_2\Gamma\left(\frac{1}{\kappa},1\right)\right)(r-t)}\ \ \text{for}\ \ t\in [r,\theta].
\end{equation}
Plugging \eqref{e:ESTV1} into \eqref{e:ESTU1} we derive
$$|u(t)|_{N(t)}\leq e^{\rho(\theta-t)}|u(\theta)|_{N(\theta)}+L_1\Gamma\left(\frac{1}{\kappa},1\right)\int_{t}^{\theta}e^{\rho(s-t)}e^{\left(\gamma-L_2\Gamma\left(\frac{1}{\kappa},1\right)\right)(r-s)}ds|v(r)|_{S(r)}\ \text{for}\ t\in[r,\theta].$$
Taking $t=r$ we get
\begin{align*}
	&|u(r)|_{N(t)}\leq e^{\rho(\theta-r)}|u(\theta)|_{N(\theta)}+L_1\Gamma\left(\frac{1}{\kappa},1\right)\int_{r}^{\theta}e^{\left(\gamma-\rho-L_2\Gamma\left(\frac{1}{\kappa},1\right)\right)(r-s)}ds|v(r)|_{S(r)}\\
	& \ \ =e^{\rho(\theta-r)}|u(\theta)|_{N(\theta)}+\frac{L_1\Gamma\left(\frac{1}{\kappa},1\right)}{\gamma-\rho-L_2\Gamma\left(\frac{1}{\kappa},1\right)}\left(1-e^{\left(\gamma-\rho-L_2\Gamma\left(\frac{1}{\kappa},1\right)\right)(r-\theta)}\right)|v(r)|_{S(r)}.
\end{align*}
Since $|v(r)|_{S(r)}=\kappa|u(r)|_{N(r)}$, we obtain, by definition of $\kappa$, 
\begin{equation}\label{eq:u_v}
\left(\frac{1}{\kappa}-\frac{L_1\Gamma\left(\frac{1}{\kappa},1\right)}{\gamma-\rho-L_2\Gamma\left(\frac{1}{\kappa},1\right)}\left(1-e^{\left(\gamma-\rho-L_2\Gamma\left(\frac{1}{\kappa},1\right)\right)(r-\theta)}\right)\right)|v(r)|_{S(r)} \leq e^{\rho(\theta-r)}|u(\theta)|_{N(\theta)}.
\end{equation}
It follows by \eqref{e:ESTV1} written for $t=\theta$ that 
\begin{equation}\label{e:ESTAA}
\left(1-\frac{L_1\kappa \Gamma\left(\frac{1}{\kappa},1\right)}{\gamma-\rho-L_2\Gamma\left(\frac{1}{\kappa},1\right)}\left(1-e^{\left(\gamma-\rho-L_2\Gamma\left(\frac{1}{\kappa},1\right)\right)(r-\theta)}\right)\right)|v(\theta)|_{S(\theta)} \leq e^{\left(\gamma-\rho-L_2\Gamma\left(\frac{1}{\kappa},1\right)\right)(r-\theta)}\kappa |u(\theta)|_{N(\theta)}.
\end{equation}
Observe that
$$
\gamma-\rho-L_2\Gamma\left(\frac{1}{\kappa},1\right) = \gamma-\rho - (\gamma-\sigma)\Gamma\left(\frac{L_1}{\sigma-\rho},\frac{L_2}{\gamma-\sigma}\right) >0,
$$
where the last inequality follows from the fact that $\rho<\sigma<\gamma$, the gap condition \eqref{e:GAPGENERAL00} and \eqref{e:DEFNKAPPA}. Moreover, the same gap condition implies that
$$
0<\frac{L_1\kappa \Gamma\left(\frac{1}{\kappa},1\right)}{\gamma-\rho-L_2\Gamma\left(\frac{1}{\kappa},1\right)} = \frac{(\sigma-\rho)\Gamma\left(\frac{L_1}{\sigma-\rho},\frac{L_2}{\gamma-\sigma}\right)}{\gamma-\rho-(\gamma-\sigma)\Gamma\left(\frac{L_1}{\sigma-\rho},\frac{L_2}{\gamma-\sigma}\right)} < 1.
$$
We deduce that the parenthesis on the left-hand side of \eqref{e:ESTAA} is positive. Therefore, the fact that $\zeta(\theta)\geq 0$, i.e., $|v(\theta)|_{S(\theta)}\geq \kappa |u(\theta)|_{N(\theta)}$, implies that
$$
\left(1-\frac{L_1\kappa \Gamma\left(\frac{1}{\kappa},1\right)}{\gamma-\rho-L_2\Gamma\left(\frac{1}{\kappa},1\right)}\left(1-e^{\left(\gamma-\rho-L_2\Gamma\left(\frac{1}{\kappa},1\right)\right)(r-\theta)}\right)\right)\kappa|u(\theta)|_{N(\theta)} \leq e^{\left(\gamma-\rho-L_2\Gamma\left(\frac{1}{\kappa},1\right)\right)(r-\theta)}\kappa |u(\theta)|_{N(\theta)}.
$$
Hence
$$
\left(1-\frac{L_1\kappa \Gamma\left(\frac{1}{\kappa},1\right)}{\gamma-\rho-L_2\Gamma\left(\frac{1}{\kappa},1\right)}\right)\left(1-e^{\left(\gamma-\rho-L_2\Gamma\left(\frac{1}{\kappa},1\right)\right)(r-\theta)}\right)\kappa|u(\theta)|_{N(\theta)} \leq 0.
$$
This means that $u(\theta)=0$, and, by \eqref{eq:u_v} $v(r)=0$, a contradiction. 
\end{proof}
Second, an analogous result concerns the situation when $\zeta(t)$ is nonpositive on a time interval. In such a case, it must be strictly negative at its end. 

\begin{lem}\label{lem:NEG}
If the gap condition \eqref{e:GAPGENERAL00} holds then for any $r<\theta$, $r,\theta\in\calt$ if only $u(\theta)\neq 0$ or $v(r)\neq 0$ then
$$\zeta(t) \leq 0 \ \ \text{ for }\ \ t\in [r,\theta]\ \ \text{ implies }\ \ \zeta(\theta)<0.$$ 
\end{lem}

\begin{proof}
Proceeding analogously as in the proof of \eqref{e:ESTU1}, we obtain
$$|u(t)|_{N(t)}\leq e^{\rho(\theta-t)}|u(\theta)|_{N(\theta)}+L_1\Gamma(1,\kappa)\int_{t}^{\theta}e^{\rho(s-t)}|u(s)|_{N(s)}ds\ \ \text{for}\ \ t\in [r,\theta].$$
By the Gronwall inequality we get
\begin{equation}\label{e:ESTU2}
|u(t)|_{N(t)}\leq e^{(\rho+L_1\Gamma(1,\kappa))(\theta-t)}|u(\theta)|_{N(\theta)}\ \ \text{for}\ \ t\in [r,\theta].
\end{equation} 
On the other hand, analogously as in the proof of \eqref{e:secondbound}, we obtain the bound
$$|v(t)|_{S(t)}\leq e^{\gamma(r-t)}|v(r)|_{S(r)}+L_2\Gamma(1,\kappa)\int_{r}^{t}e^{\gamma(s-t)}|u(s)|_{N(s)}ds\ \ \text{for}\ \ t\in [r,\theta].$$
Combining this with \eqref{e:ESTU2}, we deduce
$$|v(t)|_{S(t)}\leq e^{\gamma(r-t)}|v(r)|_{S(r)}+L_2\Gamma(1,\kappa)\int_{r}^{t}e^{\gamma(s-t)}e^{(\rho+L_1\Gamma(1,\kappa))(\theta-s)}ds|u(\theta)|_{N(\theta)}\ \ \text{for}\ \ t\in [r,\theta],$$
which for $t=\theta$ yields
\begin{align*}
	&|v(\theta)|_{S(\theta)}\leq e^{\gamma(r-\theta)}|v(r)|_{S(r)}+L_2\Gamma(1,\kappa)\int_{r}^{\theta}e^{(\gamma-\rho-L_1\Gamma(1,\kappa))(s-\theta)}ds|u(\theta)|_{N(\theta)}\\
&\ \ \ \ =e^{\gamma(r-\theta)}|v(r)|_{S(r)}+\frac{L_2\Gamma(1,\kappa)}{\gamma-\rho-L_1\Gamma(1,\kappa)}\left(1-e^{(\gamma-\rho-L_1\Gamma(1,\kappa))(r-\theta)}\right)|u(\theta)|_{N(\theta)}.
\end{align*}
Since
\begin{equation}\label{eq:v_u_2}
|v(r)|_{S(r)}\leq\kappa |u(r)|_{N(r)}\leq \kappa e^{(\rho+L_1\Gamma(1,\kappa))(\theta-r)}|u(\theta)|_{N(\theta)},    
\end{equation}
it follows that
$$|v(\theta)|_{S(\theta)}\leq\Bigl( e^{(\gamma-\rho-L_1\Gamma(1,\kappa))(r-\theta)}+\frac{L_2\Gamma(1,\kappa)}{\gamma-\rho-L_1\Gamma(1,\kappa)}(1-e^{(\gamma-\rho-L_1\Gamma(1,\kappa))(r-\theta)})\Bigr)\kappa|u(\theta)|_{N(\theta)}.$$
Analogously as in the proof of Lemma \ref{lem:POS} we have by \eqref{e:GAPGENERAL00} and \eqref{e:DEFNKAPPA}
$$
\gamma-\rho - L_1\Gamma(1,\kappa) = \gamma-\rho-(\sigma-\rho)\Gamma\left(\frac{L_1}{\sigma-\rho},\frac{L_2}{\gamma-\sigma}\right) > 0,
$$
and
$$
0< \frac{L_2\Gamma(\frac{1}{\kappa},1)}{\gamma-\rho-L_1\Gamma(1,\kappa)}  = \frac{(\gamma-\sigma)\Gamma\left(\frac{L_1}{\sigma-\rho},\frac{L_2}{\gamma-\sigma}\right)}{\gamma-\rho-(\sigma-\rho)\Gamma\left(\frac{L_1}{\sigma-\rho},\frac{L_2}{\gamma-\sigma}\right)} < 1.
$$
Now if $v(r)\neq 0$ then by \eqref{eq:v_u_2} we must also have $u(\theta)\neq0$.
Hence, since $u(\theta)\neq 0$, we get $|v(\theta)|_{S(\theta)}<\kappa |u(\theta)|_{N(\theta)}$, which proves the claim. 
\end{proof}

 We use the above two lemmas to obtain the following results.
\begin{prop}\label{prop:zeta}
	Assume the gap condition \eqref{e:GAPGENERAL00} and let the solutions $z, \tilde{z}$ of \eqref{e:VCF} on  $\calt=(-\infty,T_2]$  be such that 
 $Q(T_2)z(T_2)\neq Q(T_2)\tilde z(T_2)$.
    Then the function $\zeta\colon\calt\to \R$ has at most one zero in $\calt$. If $s$ is the unique zero of $\zeta$ in $\calt$, then we must have $\zeta(r)>0$ for $r<s$, and $\zeta(r)<0$ for $r\in (s,T_2]$.
\end{prop}

\begin{proof}
	If $s\in \calt$ is a zero of $\zeta$, then we must have $u(s)\neq 0$, because otherwise we would have $u(s) = v(s) = 0$, which means that $z(s) = \tilde z(s)$ and, by uniqueness, $z(T_2) = \tilde z(T_2)$. For contradiction, suppose that $\zeta$ has at least two zeros in $\calt$. Denote them by $s_1, s_2$. If $\zeta(s) = 0$ for every $s\in [s_1, s_2]$, we have contradiction with Lemma \ref{lem:POS}. 
	If $\zeta(s_0)>0$ for some $s_0\in [s_1,s_2]$, then let $[r_1,r_2]\subset [s_1,s_2]$ be a maximal interval containing $s_0$ such that $\zeta(r)>0$ for every $r\in (r_1,r_2)$. We must have $\zeta(r_1) = \zeta(r_2)=0$ and we have a contradiction with Lemma \ref{lem:POS}. On the other hand,  if 	$\zeta(s_0)<0$ for some $s_0\in [s_1,s_2]$, then let $[r_1,r_2]\subset [s_1,s_2]$ be a maximal interval containing $s_0$ such that $\zeta(r)<0$ on $(r_1,r_2)$. We must have $\zeta(r_1)=\zeta(r_2) = 0$ and we have a contradiction with Lemma \ref{lem:NEG}. This means that if $\zeta(s) = 0$ the function $\zeta$ cannot change sign on $(-\infty,s)$ and on $(s,T_2]$.  If $\zeta(r)<0$ for $r<s$, using the fact that $u(s)\neq 0$,  by Lemma \ref{lem:NEG} we must have $\zeta(s)<0$, a~contradiction. On the other hand, if $\zeta(r)>0$ on $(s,T_2]$, since $u(T_2)\neq 0$ we can use Lemma \ref{lem:POS}, which implies that $\zeta(s)>0$, again a contradiction. 
\end{proof}
The next result will be useful later in Section \ref{sec:stable}. Note that if we assumed the backward uniqueness of the solution, the proof would very simply follow the lines of the previous proposition. However, since we are not assuming backward uniqueness, we need some additional effort in the proof.  
\begin{prop}\label{prop:zeta_2}
	Assume the gap condition \eqref{e:GAPGENERAL00} and let the two solutions $z, \tilde{z}$ of \eqref{e:VCF} be defined on $\calt=[T_1,\infty)$ and satisfy $(I-Q(T_1))z(T_1)\neq (I-Q(T_1))\tilde z(T_1)$.
    Then, for the function $\zeta\colon\calt\to\R$ defined in \eqref{e:DEFNZETA}, exactly one of the three cases holds:
    \begin{itemize}
        \item[(i)] the function $\zeta$ has no zeros in $\calt$,
        \item[(ii)] there exists $s\geq T_1$ such that  $\zeta(r)>0$ for $r\in [T_1,s)$, $\zeta(s)=0$, and $\zeta(r)<0$ for $r\in (s,\infty)$,
        \item[(iii)] there exists $s>T_1$ such that $\zeta(r)>0$ for $r\in [T_1,s)$ and $z(r)=\tilde z(r)$ for $r\in [s,\infty)$. 
    \end{itemize} 
\end{prop}

\begin{proof}
	Consider $\calt_1\subset \calt$, the set of  zeros of $\zeta$ such that $u(s)= 0$ or $v(s)=0$ for $s\in \calt_1$. If this set is nonempty, then for $s\in \calt_1$ we must have have both $u(s)=0$ and $v(s)=0$, and, in consequence, $z(s) = \tilde z(s)$.  In such case, by uniqueness, $z(r) = \tilde z(r)$ for every $r\geq s$. This means that either  $\calt_1$ is empty or $\calt_1 = [\overline s,\infty)$ for certain $\overline s\geq T_1$ and $z$ coincides with $\tilde z$ on $\calt_1$. Assume first that $\calt_1$ is empty. An argument that follows the lines of the proof of the previous lemma shows that either assertion (i) or assertion (ii) must hold.  If $\calt_1$ is nonempty also $\calt\setminus\calt_1$ must be nonempty, since $v(T_1)\neq 0$. Thus $\overline{s}>T_1$. Suppose that there exists $\overline{r}\in \calt\setminus \calt_1$ such that $\zeta(\overline{r})=0$. If $\zeta(r)=0$ for every $r\in [\overline{r},\overline{s}]$, we have a~contradiction with Lemma \ref{lem:POS}. If $\zeta(s_0)<0$ for some $s_0\in [\overline{r},\overline{s}]$ then, taking the maximal interval $[s_1,s_2]$ containing $s_0$ such that $\zeta(s)<0$ for $s\in (s_1,s_2)$ we must have $v(s_1)\neq 0$ and $\zeta(s_2)=0$ which gives a contradiction by Lemma \ref{lem:NEG}. On the the other hand, if $\zeta(s_0)>0$ for some $s_0\in [\overline{r},\overline{s}]$ then, taking the maximal interval $[s_1,s_2]$ containing $s_0$ such that $\zeta(s)>0$ for $s\in (s_1,s_2)$ we must  have $v(s_1)\neq 0$ and $\zeta(s_1)=0$, a contradiction with Lemma \ref{lem:POS}. This means that $\zeta$ is nonzero and does not change sign on $[T_1,\overline{s})$. If $\zeta(r)<0$ on  $[T_1,\overline{s}),$ then, as $v(T_1)\neq 0$ and $\zeta(\overline{s})=0$ we get a~contradiction with Lemma \ref{lem:NEG}. This implies that we must have the assertion (iii). 
\end{proof}

Previous results were valid for any two solutions of the considered evolution process. We come back to the manifold that we have constructed. In fact, we are now in a position to get the Lipschitzness of $\Sigma$. To this end, let $\eta,\tilde\eta\in X$ and $\tau\in\R$. We denote $z=z_{\tau,\eta}$, $\tilde{z}=z_{\tau,\tilde\eta}$ on $\calt=(-\infty,\tau]$, the unique fixed points in $E_{\sigma,\tau}$ of $H_{\tau,\eta}$ and $H_{\tau,\tilde\eta}$, respectively. Below we show that $\zeta$ defined in \eqref{e:DEFNZETA} must be negative in $(-\infty,\tau)$.

\begin{cor}\label{cor:ALWAYSNEGATIVE}
If the gap condition \eqref{e:GAPGENERAL00} holds and $Q(\tau)\eta\neq Q(\tau)\tilde\eta$ then 
$\zeta(t)<0$ for all $t<\tau$.
\end{cor}

\begin{proof}
The function $\zeta$ is a real-valued continuous function on $(-\infty,\tau]$. In particular, by Proposition~\ref{prop:zeta} exactly one of the following situations takes place:
\begin{itemize}
\item[(i)] function $\zeta$ is negative in $(-\infty,\tau)$,
\item[(ii)] function $\zeta$ is positive in $(-\infty,\tau)$,
\item[(iii)] function $\zeta$ has a unique zero $t_0$ in $(-\infty,\tau)$ and
then $\zeta(t)>0$ for $t<t_0$ and $\zeta(t)<0$ for $t_0<t<\tau$. 
\end{itemize} 
We will show that (ii) and (iii) lead to a contradiction with the choice of $\sigma$. In both these cases there exists $t_1<\tau$ such that $\zeta(r)>0$ for $r\leq t_1$. We have
$$v(t)=L(t,r)v(r)+\int_{r}^{t}L(t,s)(I-Q(s))[f(s,z(s))-f(s,\tilde z(s))]ds\ \ \text{for}\ \ -\infty<r\leq t\leq t_1,$$
and
\begin{align*}
|v(t)|_{S(t)}&\leq e^{\gamma(r-t)}|v(r)|_{S(r)}+L_2\int_{r}^{t}e^{\gamma(s-t)}\Gamma(|u(s)|_{N(s)},|v(s)|_{S(s)})ds\\
&\leq e^{\gamma(r-t)}|v(r)|_{S(r)}+L_2\Gamma\left(\frac{1}{\kappa},1\right)\int_{r}^{t}e^{\gamma(s-t)}|v(s)|_{S(s)}ds\ \ \text{for}\ \ r\leq t\leq t_1.
\end{align*}
By the Gronwall inequality we get
$$|v(t)|_{S(t)}\leq |v(r)|_{S(r)}e^{\left(\gamma-L_2\Gamma\left(\frac{1}{\kappa},1\right)\right)(r-t)}\ \ \text{for}\ \ r\leq t\leq t_1.$$
We have
\begin{align*}
	&|v(t)|_{S(t)}\leq |v(r)|_{S(r)}e^{\left(\gamma-L_2\Gamma\left(\frac{1}{\kappa},1\right)\right)(r-t)}=\frac{1}{\Gamma(0,1)}\Gamma(0,|v(r)|_{S(r)})e^{\left(\gamma-L_2\Gamma\left(\frac{1}{\kappa},1\right)\right)(r-t)}\\
& \ \ \leq\frac{1}{\Gamma(0,1)}\norm{z(r)-\tilde z(r)}_{r}e^{\left(\gamma-L_2\Gamma\left(\frac{1}{\kappa},1\right)\right)(r-t)}\leq \frac{1}{\Gamma(0,1)}\norm{z-\tilde z}_{E_{\sigma,t}}e^{\left(\gamma-\sigma-L_2\Gamma\left(\frac{1}{\kappa},1\right)\right)(r-t)},\ r\leq t\leq t_1.
\end{align*}
Passing with $r\to-\infty$ we get $|v(t)|_{S(t)}=0$, since $\gamma-\sigma-L_2\Gamma\left(\frac{1}{\kappa},1\right)>0$.
This implies that $|u(t)|_{N(t)}<\frac{1}{\kappa}|v(t)|_{S(t)}=0$, which is not possible. 
\end{proof}

Thus we always have $\zeta(t)\leq 0$ for $t\in (-\infty,\tau]$. We deduce the following theorem.
\begin{thm}
	The function $\Sigma$ that defined the manifold $\{\calm(t)\colon t\in \R\}$ is Lipschitz continuous, 
		\begin{equation}\label{e:SIGMALIP}
			|\Sigma(\tau,\eta)-\Sigma(\tau,\tilde\eta)|_{S(\tau)}\leq\kappa|Q(\tau)(\eta-\tilde\eta)|_{N(\tau)}\ \ \text{for}\ \ \eta,\tilde\eta\in X,\ \tau\in\R,
		\end{equation}	
	with the Lipschitz constant $\kappa = \frac{L_2}{L_1}\frac{\sigma-\rho}{\gamma-\sigma}$.
\end{thm}
\begin{proof}
Since the inequality \eqref{e:SIGMALIP} holds trivially when $Q(\tau)\eta=Q(\tau)\tilde\eta$, we consider $Q(\tau)\eta\neq Q(\tau)\tilde\eta$. 	Since $\zeta(\tau) \leq 0$, it follows that $|v(\tau)|_{S(\tau)}\leq \kappa |u(\tau)|_{N(\tau)}$. This further means that 
	$$|(I-Q(\tau))(z(\tau)-\tilde z(\tau)|_{S(\tau)}\leq \kappa |Q(\tau)(z(\tau)-\tilde z(\tau))|_{N(\tau)},$$ 
which directly implies the bound \eqref{e:SIGMALIP}. 
\end{proof}
As it is seen in the next result, the Lipschitz constant of the inertial manifold can be proved to be lower than it follows from the previous theorem. 
\begin{thm}\label{thm:betterlip}
The function $\Sigma$ defining the manifold $\{ \calm(t)\colon t\in \R\}$ is Lipschitz continuous,
\begin{equation*}
|\Sigma(\tau,\eta)-\Sigma(\tau,\tilde\eta)|_{S(\tau)}\leq\kappa_1|Q(\tau)(\eta-\tilde\eta)|_{N(\tau)},\ \eta,\tilde\eta\in X,\ \tau\in\R,
\end{equation*}
with the Lipschitz constant
\begin{equation*}
\kappa_1=\frac{L_2\Gamma(1,\kappa)}{\gamma-\rho-L_1\Gamma(1,\kappa)}<\kappa.
\end{equation*}
\end{thm}

\begin{proof}
Analogously as in the proof of Lemma \ref{lem:NEG}  we have
$$|u(t)|_{N(t)}\leq e^{\rho(\tau-t)}|Q(\tau)(\eta-\tilde\eta)|_{N(\tau)}+L_1\int_{t}^{\tau}e^{\rho(s-t)}\norm{z(s)-\tilde{z}(s)}_{s}ds\ \ \text{for}\ \ t\leq\tau.$$
Since by Lemma \ref{lem:characterization} $z(s)$, $\tilde{z}(s)$ are in the $\kappa$-Lipschitz manifold $\calm(s)$ and $v(s)=\Sigma(s,z(s))-\Sigma(s,\tilde{z}(s))$, we get
$$|u(t)|_{N(t)}\leq e^{\rho(\tau-t)}|Q(\tau)(\eta-\tilde\eta)|_{N(\tau)}+L_1\Gamma(1,\kappa)\int_{t}^{\tau}e^{\rho(s-t)}|u(s)|_{N(s)}ds\ \ \text{for}\ \ t\leq\tau$$
and the Gronwall inequality gives
\begin{equation}\label{e:ESTU}
|u(t)|_{N(t)}\leq |Q(\tau)(\eta-\tilde\eta)|_{N(\tau)}e^{(\rho+L_1\Gamma(1,\kappa))(\tau-t)}\ \ \text{for}\ \  t\leq\tau.
\end{equation}
We also have
$$|\Sigma(\tau,\eta)-\Sigma(\tau,\tilde\eta)|_{S(\tau)}\leq L_2\int_{-\infty}^{\tau}e^{\gamma(s-\tau)}\norm{z(s)-\tilde{z}(s)}_{s}ds.$$
As $z(s)$, $\tilde{z}(s)$ are in the manifold $\calm(s)$, which is $\kappa$-Lipschitz, we get
$$|\Sigma(\tau,\eta)-\Sigma(\tau,\tilde\eta)|_{S(\tau)}\leq L_2\Gamma(1,\kappa)\int_{-\infty}^{\tau}e^{\gamma(s-\tau)}|u(s)|_{N(s)}ds.$$
Plugging \eqref{e:ESTU} in the above bound, we obtain
\begin{align*}
	&|\Sigma(\tau,\eta)-\Sigma(\tau,\tilde\eta)|_{S(\tau)}\leq L_2\Gamma(1,\kappa)|Q(\tau)(\eta-\tilde\eta)|_{N(\tau)}\int_{-\infty}^{\tau}e^{(\gamma-\rho-L_1\Gamma(1,\kappa))(s-\tau)}ds\\
& \ \ =\frac{L_2\Gamma(1,\kappa)}{\gamma-\rho-L_1\Gamma(1,\kappa)}|Q(\tau)(\eta-\tilde\eta)|_{N(\tau)}:=\kappa_1 |Q(\tau)(\eta-\tilde\eta)|_{N(\tau)}.
\end{align*}
Straightforward computation shows that $\frac{L_2\Gamma(1,\kappa)}{\gamma-\rho-L_1\Gamma(1,\kappa)} < \kappa$, which ends the proof.
\end{proof}

\begin{cor}\label{cor:betterlip}
The function $\Sigma$ defining the manifold $\{ \calm(t)\colon t\in \R\}$ is Lipschitz continuous,
\begin{equation}\label{e:SIGMALIPBETTER}
|\Sigma(\tau,\eta)-\Sigma(\tau,\tilde\eta)|_{S(\tau)}\leq\kappa_\Sigma|Q(\tau)(\eta-\tilde\eta)|_{N(\tau)},\ \eta,\tilde\eta\in X,\ \tau\in\R,
\end{equation}
with the Lipschitz constant $0<\kappa_\Sigma<\kappa$ satisfying
\begin{equation}\label{e:KAPPABEST}
\gamma-\rho=L_1\Gamma(1,\kappa_\Sigma)+L_2\Gamma\left(\frac{1}{\kappa_\Sigma},1\right),
\end{equation}
which is equivalent to say that $\kappa_\Sigma = \frac{L_2}{L_1}\frac{\sigma-\rho}{\gamma-\sigma}$ for $\sigma\in (\rho,\gamma)$ satisfying $\Gamma\left(\frac{L_1}{\sigma-\rho},\frac{L_2}{\gamma-\sigma}\right) = 1.$
\end{cor}

\begin{proof}
The proof of Theorem~\ref{thm:betterlip} can be used to iteratively improve the Lipschitz constant. Namely, 
$$\kappa_2=\frac{L_2\Gamma(1,\kappa_1)}{\gamma-\rho-L_1\Gamma(1,\kappa_1)} < \kappa_1$$
is a valid and better Lipschitz constant for $\Sigma$. Repeating the procedure, we obtain a strictly decreasing sequence $(\kappa_n)$ of 
Lipschitz constants converging to the Lipschitz constant $\kappa_\Sigma>0$ satisfying \eqref{e:KAPPABEST}. 
\end{proof}

Observe that the equation \eqref{e:KAPPABEST} does not contain $\sigma$ from the gap condition \eqref{e:GAPGENERAL00}. Hence, the value of $\kappa_\Sigma$ depends only on $\gamma, \rho, L_1, L_2$ and the chosen admissible norm $\Gamma$.     
In the next sections we always assume that $\kappa_\Sigma$ is the Lipschitz constant of the constructed inertial manifold.	

\subsection{Estimate of the distance to the manifold and its controlled growth.} We have proved that the constructed manifold is invariant and consists of Lipschitz graphs. We recall that if $(\tau,\eta) \in \R\times X$ and $z_{\tau,\eta} \in E_{\sigma,\tau}$ is a fixed point of $H_{\tau,\eta}$, then, denoting $p(t) = (I-Q(t))z_{\tau,\eta}(t)$ and $q(t)=Q(t)z_{\tau,\eta}(t)$ for $t\in (-\infty,\tau]$, by Lemma \ref{lem:characterization} and the definition of $\Sigma$, we deduce $p(t) = \Sigma(t,q(t))$. Moreover, applying projections $Q(t)$ and $I-Q(t)$ to \eqref{eq:fixed_point} we deduce that $p, q$ satisfy the following equations:
\begin{align}\label{e:QWITHSIGMA}
&	q(t)=L(t,\tau)Q(\tau)\eta-\int_{t}^{\tau}L(t,s)Q(s)f(s,q(s)+\Sigma(s,q(s)))ds\ \ \text{for}\ \  t\in (-\infty,\tau],\\
& \label{e:PWITHSIGMA}
	\Sigma(t,q(t))=\int_{-\infty}^{t}L(t,s)(I-Q(s))f(s,q(s)+\Sigma(s,q(s)))ds\ \  \text{for}\ \ t\in\ \  (-\infty,\tau].
\end{align}

We continue by showing more properties of $\{ \calm(t)\colon t\in \R\}$ defined in \eqref{e:MANIFOLD}, stressing that the proofs of this subsection are largely analogous to those of \cite{CLMO-S25}. 

Given $(\tau,\eta)\in\R\times X$, we define the quantity
$$\chi(t)=T(t,\tau)\eta-P_\Sigma(t)T(t,\tau)\eta=(I-Q(t))T(t,\tau)\eta-\Sigma(t,Q(t)T(t,\tau)\eta)\in S(t)\ \ \text{for}\ \  t\geq\tau,$$
where $P_\Sigma(t)\colon X\to\calm(t)$ for $t\in\R$ is defined in \eqref{eq-Def_P_Sigma}.

We carry out the fixed point procedure  from $T(t,\tau)\eta$ which yields the fixed point $z_{t,T(t,\tau)\eta} \in E_{\sigma,\tau}$. 
For $t\geq\tau$ we denote
$$q_{*}(s,t)=Q(s)z_{t,T(t,\tau)\eta}(s)\ \ \text{for}\ \ s\in (-\infty,t].$$
In particular, for $t=\tau$ and $s\in (-\infty,\tau]$ we have $q_{*}(s,\tau)=q(s)$ for $s\in (-\infty,\tau]$. We start from deriving two auxiliary estimates that will be useful in the estimation of the distance of solutions to the invariant manifold.

\begin{lem}
We have
\begin{align}\label{e:QSTAR1}
& |q_{*}(s,t)-Q(s)T(s,\tau)\eta|_{N(s)}\leq L_1\Gamma(0,1)\int_{s}^{t}|\chi(r)|_{S(r)}e^{(\rho+L_1\Gamma(1,\kappa_\Sigma))(r-s)}dr\ \ \text{for}\ \  s\in[\tau,t],\\
& \label{e:QSTAR2}
|q_{*}(s,t)-q_{*}(s,\tau)|_{N(s)}\leq L_1\Gamma(0,1)\int_{\tau}^{t} |\chi(r)|_{S(r)}e^{(\rho+L_1\Gamma(1,\kappa_\Sigma))(r-s)}dr\ \ \text{for}\ \  s\in (-\infty,\tau].
\end{align}
\end{lem}

\begin{proof}
The function $q_{*}(s,t)$ defined for $s\in (-\infty,t]$ satisfies the equation
\begin{align*}
	& q_{*}(s,t)=L(s,t)Q(t)T(t,\tau)\eta-\int_{s}^{t}L(s,r)Q(r)f(r,q_{*}(r,t)+\Sigma(r,q_{*}(r,t)))dr\\
& =L(s,\tau)Q(\tau)\eta+\int_{\tau}^{t}L(s,r)Q(r)f(r,T(r,\tau)\eta)dr-\int_{s}^{t}L(s,r)Q(r)f(r,q_{*}(r,t)+\Sigma(r,q_{*}(r,t)))dr.
\end{align*}
Therefore, for $s\in[\tau,t]$ we have
$$q_{*}(s,t)-Q(s)T(s,\tau)\eta=\int_{s}^{t}L(s,r)Q(r)[f(r,T(r,\tau)\eta)-f(r,q_{*}(r,t)+\Sigma(r,q_{*}(r,t)))]dr.$$
We continue by estimating for $s\in[\tau,t]$ 
\begin{align*}
&	e^{\rho s}|q_{*}(s,t)-Q(s)T(s,\tau)\eta|_{N(s)}\\
& \ \ \leq L_1\int_{s}^{t}e^{\rho r}\norm{Q(r)T(r,\tau)\eta+(I-Q(r))T(r,\tau)\eta-q_{*}(r,t)-\Sigma(r,q_{*}(r,t)))}_{r}dr\\
& =L_1\int_{s}^{t}e^{\rho r}\Gamma(|Q(r)T(r,\tau)\eta-q_{*}(r,t)|_{N(r)},|(I-Q(r))T(r,\tau)\eta-\Sigma(r,q_{*}(r,t))|_{S(r)})\, dr.
\end{align*}
But we have
\begin{align*}
& |(I-Q(r))T(r,\tau)\eta-\Sigma(r,q_{*}(r,t))|_{S(r)}\\
&  \leq  |(I-Q(r))T(r,\tau)\eta-\Sigma(r,Q(r)T(r,\tau)\eta)|_{S(r)}+|\Sigma(r,Q(r)T(r,\tau)\eta)-\Sigma(r,q_{*}(r,t))|_{S(r)}\\
& \leq |\chi(r)|_{S(r)} + \kappa_\Sigma |Q(r)T(r,\tau)\eta-q_{*}(r,t)|_{N(r)}.
\end{align*}
Hence we obtain
\begin{align*}
	&	e^{\rho s}|q_{*}(s,t)-Q(s)T(s,\tau)\eta|_{N(s)}\\
& \leq L_1\int_{s}^{t}e^{\rho r}\Gamma(|Q(r)T(r,\tau)\eta-q_{*}(r,t)|_{N(r)},|\chi(r)|_{S(r)}+\kappa_\Sigma|Q(r)T(r,\tau)\eta-q_{*}(r,t)|_{N(r)})dr\\
&\leq\int_{s}^{t} L_1\Gamma(0,1) e^{\rho r}|\chi(r)|_{S(r)}+ L_1\Gamma(1,\kappa_\Sigma)e^{\rho r}|q_{*}(r,t)-Q(r)T(r,\tau)\eta|_{N(r)}dr.
\end{align*}
By the Gronwall inequality (see Lemma~\ref{lem:GRONWALL2}) we get \eqref{e:QSTAR1}. To prove \eqref{e:QSTAR2} observe that for $s\in (-\infty,\tau]$ we have
$$q_{*}(s,t)=L(s,\tau)q_{*}(\tau,t)-\int_{s}^{\tau}L(s,r)Q(r)f(r,q_{*}(r,t)+\Sigma(r,q_{*}(r,t)))dr,$$
and
\begin{align*}
	&q_{*}(s,t)-q_{*}(s,\tau)=L(s,\tau)[q_{*}(\tau,t)-Q(\tau)\eta]\\
& \ \ \ \ -\int_{s}^{\tau}L(s,r)Q(r)[f(r,q_{*}(r,t)+\Sigma(r,q_{*}(r,t)))-f(r,q_{*}(r,\tau)+\Sigma(r,q_{*}(r,\tau)))]dr.
\end{align*}
Thus we get for $s\in (-\infty,\tau]$
\begin{align*}
&	e^{\rho s}|q_{*}(s,t)-q_{*}(s,\tau)|_{N(s)}\leq e^{\rho\tau}|q_{*}(\tau,t)-Q(\tau)\eta|_{N(\tau)}\\
& \qquad\qquad  +L_1\int_{s}^{\tau}e^{\rho r}\norm{q_{*}(r,t)+\Sigma(r,q_{*}(r,t))-q_{*}(r,\tau)-\Sigma(r,q_{*}(r,\tau))}_{r}dr\\
& \  =e^{\rho\tau}|q_{*}(\tau,t)-Q(\tau)\eta|_{N(\tau)}\\
& \qquad \qquad +L_1\int_{s}^{\tau}e^{\rho r}\Gamma(|q_{*}(r,t)-q_{*}(r,\tau)|_{N(r)},|\Sigma(r,q_{*}(r,t))-\Sigma(r,q_{*}(r,\tau))|_{S(r)})dr\\
& \ \leq L_1\Gamma(0,1)\int_{\tau}^{t}e^{\rho r}|\chi(r)|_{S(r)}e^{L_1\Gamma(1,\kappa_\Sigma)(r-\tau)}dr+L_1\Gamma(1,\kappa_\Sigma)\int_{s}^{\tau}e^{\rho r}|q_{*}(r,t)-q_{*}(r,\tau)|_{N(r)}dr,
\end{align*}
where, in the last estimate, we have used \eqref{e:QSTAR1}. By the Gronwall lemma we obtain \eqref{e:QSTAR2}.
\end{proof}
We are in position to derive the bound on the evolution of the  error between the solution of the problem and its nonlinear projection $P_\Sigma$ onto the invariant manifold.
\begin{lem}\label{lem:error}
We have
\begin{equation}\label{e:ESTFORWARD}
|T(t,\tau)\eta-P_\Sigma(t)T(t,\tau)\eta|_{S(t)}\leq |\eta-P_\Sigma(\tau)\eta|_{S(\tau)}e^{\bigl(-\gamma+\frac{(\gamma-\rho)L_2\Gamma(0,1)}{\gamma-\rho-L_1\Gamma(1,\kappa_\Sigma)}\bigr)(t-\tau)}\ \ \text{for}\ \  t\geq\tau,\ \eta\in X.
\end{equation}
\end{lem}

\begin{proof}
By \eqref{e:PWITHSIGMA} we have for $t\geq\tau$
\begin{align*}
&	\chi(t)-L(t,\tau)\chi(\tau)=(I-Q(t))[T(t,\tau)\eta-L(t,\tau)\eta]+L(t,\tau)\Sigma(\tau,Q(\tau)\eta)-\Sigma(t,Q(t)T(t,\tau)\eta)\\
& =\int_{\tau}^{t}L(t,s)(I-Q(s))f(s,T(s,\tau)\eta)ds+\int_{-\infty}^{\tau}L(t,s)(I-Q(s))f(s,q_{*}(s,\tau)+\Sigma(s,q_{*}(s,\tau))ds\\
& \ \ -\int_{-\infty}^{t}L(t,s)(I-Q(s))f(s,q_{*}(s,t)+\Sigma(s,q_{*}(s,t)))ds\\
&=\int_{\tau}^{t}L(t,s)(I-Q(s))[f(s,T(s,\tau)\eta)-f(s,q_{*}(s,t)+\Sigma(s,q_{*}(s,t)))]ds\\
&\ \ +\int_{-\infty}^{\tau}L(t,s)(I-Q(s))[f(s,q_{*}(s,\tau)+\Sigma(s,q_{*}(s,\tau)))-f(s,q_{*}(s,t)+\Sigma(s,q_{*}(s,t)))]ds.
\end{align*}
Estimating, we obtain
\begin{align*}
	& e^{\gamma t}|\chi(t)-L(t,\tau)\chi(\tau)|_{S(t)}\\
	&\ \  \leq L_2\int_{\tau}^{t}e^{\gamma s}\norm{T(s,\tau)\eta-q_{*}(s,t)-\Sigma(s,q_{*}(s,t))}_{s}ds\\
	& \qquad \qquad +L_2\Gamma(1,\kappa_\Sigma)\int_{-\infty}^{\tau}e^{\gamma s}|q_{*}(s,t)-q_{*}(s,\tau)|_{N(s)}ds\\
	& \ \ \leq L_2\Gamma(1,\kappa_\Sigma)\int_{\tau}^{t}e^{\gamma s}|Q(s)T(s,\tau)\eta-q_{*}(s,t)|_{N(s)}ds\\
& \qquad \qquad +L_2\Gamma(0,1)\int_{\tau}^{t}e^{\gamma s}|(I-Q(s))T(s,\tau)\eta-\Sigma(s,Q(s)T(s,\tau)\eta)|_{S(s)}ds\\
& \qquad \qquad +L_2\Gamma(1,\kappa_\Sigma)\int_{-\infty}^{\tau}e^{\gamma s}|q_{*}(s,t)-q_{*}(s,\tau)|_{N(s)}ds.
\end{align*}
Therefore we get by \eqref{e:QSTAR1} and \eqref{e:QSTAR2}
\begin{align*}
	& e^{\gamma t}|\chi(t)-L(t,\tau)\chi(\tau)|_{S(t)}\\
	& \leq L_2\Gamma(0,1)\int_{\tau}^{t}e^{\gamma s}|\chi(s)|_{S(s)}ds+L_1L_2\Gamma(0,1)\Gamma(1,\kappa_\Sigma)\int_{\tau}^{t}\int_{s}^{t}e^{\gamma s}|\chi(r)|_{S(r)}e^{(\rho+L_1\Gamma(1,\kappa_\Sigma))(r-s)}drds\\
&\ \ +L_1L_2\Gamma(0,1)\Gamma(1,\kappa_\Sigma)\int_{-\infty}^{\tau}\int_{\tau}^{t} e^{\gamma s}|\chi(r)|_{S(r)}e^{(\rho+L_1\Gamma(1,\kappa_\Sigma))(r-s)}drds\\
& =L_2\Gamma(0,1)\int_{\tau}^{t}e^{\gamma s}|\chi(s)|_{S(s)}ds\\
&\ \ \ \ \ \ \ +L_1L_2\Gamma(0,1)\Gamma(1,\kappa_\Sigma)\int_{\tau}^{t}|\chi(r)|_{S(r)}e^{(\rho+L_1\Gamma(1,\kappa_\Sigma))r}\int_{-\infty}^{r} e^{(\gamma -\rho-L_1\Gamma(1,\kappa_\Sigma))s}dsdr\\
&=L_2\Gamma(0,1)\int_{\tau}^{t} e^{\gamma r}|\chi(r)|_{S(r)}dr+\frac{L_1L_2\Gamma(0,1)\Gamma(1,\kappa_\Sigma)}{\gamma-\rho-L_1\Gamma(1,\kappa_\Sigma)}\int_{\tau}^{t}e^{\gamma r}|\chi(r)|_{S(r)}dr\ \ \text{for}\ \  t\geq\tau,
\end{align*}
since $\gamma-\rho> L_1\Gamma(1,\kappa_\Sigma)$.
This yields
\begin{align*}
	& e^{\gamma t}|\chi(t)|_{S(t)}\leq e^{\gamma\tau}|\chi(\tau)|_{S(\tau)}+\left(L_2 +\frac{L_1L_2\Gamma(1,\kappa_\Sigma)}{\gamma-\rho-L_1\Gamma(1,\kappa_\Sigma)}\right)\Gamma(0,1)\int_{\tau}^{t}e^{\gamma r}|\chi(r)|_{S(r)}dr\\
	& = e^{\gamma\tau}|\chi(\tau)|_{S(\tau)}+\frac{(\gamma-\rho)L_2\Gamma(0,1)}{\gamma-\rho-L_1\Gamma(1,\kappa_\Sigma)}\int_{\tau}^{t}e^{\gamma r}|\chi(r)|_{S(r)}dr \ \ \text{for}\  t\geq\tau.
	\end{align*}
The Gronwall inequality yields
$$|\chi(t)|_{S(t)}\leq |\chi(\tau)|_{S(\tau)}e^{\bigl(-\gamma+\frac{(\gamma-\rho)L_2\Gamma(0,1)}{\gamma-\rho-L_1\Gamma(1,\kappa_\Sigma)}\bigr)(t-\tau)}\ \ \text{for}\ \  t\geq\tau,$$
and the proof is complete.
\end{proof}

The next result, in addition to providing the estimate on the distance of solutions to the manifold, yields the criterion under which it is attracting, that is, when this manifold is inertial. 

\begin{cor}\label{cor:POSINV}
We have
\begin{equation}\label{e:error_repeated}
|T(t,\tau)\eta-P_\Sigma(t)T(t,\tau)\eta|_{S(t)}\leq |\eta-P_\Sigma(\tau)\eta|_{S(\tau)}e^{-\omega(t-\tau)}\ \ \text{for}\ \  t\geq\tau,\ \eta\in X,
\end{equation}
or, in the original norm of $X$,
\begin{equation}\label{e:ESTFORWARD2M}
\norm{T(t,\tau)\eta-P_\Sigma(t)T(t,\tau)\eta}\leq M(1+\kappa_\Sigma)\norm{\eta}e^{-\omega(t-\tau)}\ \ \text{for}\ \  \ t\geq\tau,\ \eta\in X,
\end{equation}
where
\begin{equation}\label{e:DEFNOMEGA}
\omega=\gamma-\frac{(\gamma-\rho)L_2\Gamma(0,1)}{\gamma-\rho-L_1\Gamma(1,\kappa_\Sigma)}
\end{equation}
with the Lipschitz constant $\kappa_\Sigma$ of $\{ \calm(t)\colon t\in \R\}$. Moreover, $\omega$ belongs to the interval $(\rho,\gamma)$.
In particular, given a bounded set $G\subset X$ there exists $c_{G}>0$ such that
\begin{equation*}
\dist(T(t,\tau)G,\calm(t))\leq c_{G}e^{-\omega(t-\tau)},\ t\geq\tau.
\end{equation*}
If $\omega>0$, then $\{\calm(t)\colon t\in\R\}$ is an inertial manifold for the process. 
\end{cor}

\begin{proof}
The bound \eqref{e:error_repeated} follows directly from Lemma \ref{lem:error}. The estimate \eqref{e:ESTFORWARD2M} follows from \eqref{e:error_repeated}, \eqref{e:EQUIVNS} and \eqref{e:QI-QONX}, since
$$|\eta-P_\Sigma(\tau)\eta|_{S(\tau)}\leq|(I-Q(\tau))\eta|_{S(\tau)}+|\Sigma(\tau,Q(\tau)\eta)-\Sigma(\tau,0)|_{S(\tau)}\leq(1+\kappa_\Sigma)M\norm{\eta}.$$
Recalling \eqref{e:KAPPABEST}, we have
$$L_1\Gamma(1,\kappa_\Sigma)+L_2\Gamma(0,1)<L_1\Gamma(1,\kappa_\Sigma)+L_2\Gamma\left(\frac{1}{\kappa_\Sigma},1\right) =\gamma-\rho.$$
This means that
$$
\omega = \gamma  - \frac{(\gamma-\rho)L_2\Gamma(0,1)}{\gamma-\rho-L_1\Gamma(1,\kappa_\Sigma)} > \gamma - \frac{(\gamma-\rho)L_2\Gamma(0,1)}{L_2\Gamma(0,1)} = \rho,
$$
and the assertion follows.
\end{proof}

The last property of $\{\calm(t)\colon t\in\R\}$ is the control of the growth of the solutions that lie on it. To this end, let $\eta=Q(\tau)\eta+\Sigma(\tau,Q(\tau)\eta)\in\calm(\tau)$. We define 
\begin{equation}\label{e:XI}
z(s)=\begin{cases}
T(s,\tau)\eta=Q(s)T(s,\tau)\eta+\Sigma(s,Q(s)T(s,\tau)\eta)&\ \text{if}\  s\geq\tau,\\
q_{\tau,\eta}(s)+\Sigma(s,q_{\tau,\eta}(s))=z_{\tau,\eta}(s)&\ \text{if}\  -\infty<s\leq\tau,
\end{cases}
\end{equation}
with $q_{\tau,\eta}(s)=Q(s)z_{\tau,\eta}(s)$ for $s\leq\tau$. Note that the formula for $T(s,\tau)\eta$, $s\geq\tau$, is a consequence of 
\eqref{e:ESTFORWARD}, since $\eta\in\calm(\tau)$.

\begin{lem}\label{lem:COMPLETETRAJECTORY}
The function $z\colon\R\to X$ given in \eqref{e:XI} is a unique global solution of \eqref{e:VCF} through $\eta\in\calm(\tau)$ within $\{\calm(t)\colon t\in\R\}$, that is,
$z(t)\in\calm(t)$ for $t\in\R$, $z(\tau)=\eta$ and
$$T(t,s)z(s)=z(t)\ \ \text{for every}\ \  t\geq s.$$
Moreover, $z$ has controlled growth in the past, that is, we have
\begin{equation}\label{e:XIBACKWARDS}
\norm{z(s)}_s\leq \Gamma(1,\kappa_\Sigma)e^{(\rho+L_1\Gamma(1,\kappa_\Sigma))(\tau-s)}|Q(\tau)\eta|_{N(\tau)}\ \ \text{for}\ \ s\in (-\infty,\tau],
\end{equation}
or, in the original norm of $X$,
\begin{equation*}
\norm{z(s)}\leq Mc_\Gamma \Gamma(1,\kappa_\Sigma)e^{(\rho+L_1\Gamma(1,\kappa_\Sigma))(\tau-s)}\norm{\eta}\ \ \text{for every}\ \  s\in (-\infty,\tau].
\end{equation*}
\end{lem}
\begin{proof}
Consider the interval $[s,t]$. If $\tau \leq s$,  then
$$T(t,s)z(s)=T(t,s)T(s,\tau)\eta=T(t,\tau)\eta=z(t).$$
On the other hand, if  $\tau\geq t$, by Lemma~\ref{lem:lemma_solution} 
we have
$$T(t,s)z(s)=z(t).$$
Finally, if  $\tau \in (s,t)$, then we have
$$T(t,s)z(s)=T(t,\tau)T(\tau,s)z(s)=T(t,\tau)z(\tau)=T(t,\tau)\eta=z(t).$$
To prove the bound, we write down the projection of $z_{\tau,\eta}(s)$ onto $N(s)$ for $s\leq\tau$ and estimate
$$e^{\rho s}|Q(s)z_{\tau,\eta}(s)|_{N(s)}\leq e^{\rho\tau}|Q(\tau)\eta|_{N(\tau)}+L_1\int_{s}^{\tau}e^{\rho r}\Gamma(|Q(r)z_{\tau,\eta}(r)|_{N(r)},|\Sigma(r,Q(r)z_{\tau,\eta}(r))|_{S(r)})dr.$$
Since $\Sigma$ is $\kappa_\Sigma$-Lipschitz, by the Gronwall inequality we get
\begin{equation*}
	|Q(s)z_{\tau,\eta}(s)|_{N(s)}\leq e^{(\rho+L_1\Gamma(1,\kappa_\Sigma))(\tau-s)}|Q(\tau)\eta|_{N(\tau)}\ \ \text{for}\ \  s\leq\tau.
\end{equation*}
We have for $s\leq\tau$
$$	\norm{z(s)}_{s}
 =\Gamma(|Q(s)z_{\tau,\eta}(s)|_{N(s)},|\Sigma(s,Q(s)z_{\tau,\eta}(s))|_{S(s)})\leq\Gamma(1,\kappa_\Sigma)|Q(s)z_{\tau,\eta}(s)|_{N(s)}.$$
Thus
$$\norm{z(s)}_s\leq\Gamma(1,\kappa_\Sigma)e^{(\rho+L_1\Gamma(1,\kappa_\Sigma))(\tau-s)}|Q(\tau)\eta|_{N(\tau)}\ \ \text{for}\ \ s\in (-\infty,\tau],$$
and the assertion follows.
\end{proof}

\begin{rem}
Note that $\rho+L_1\Gamma(1,\kappa_\Sigma)<\rho+L_1\Gamma(1,\kappa)<\sigma$, and hence the above bound gives us better control than merely the fact that the fixed point belongs to the space $E_{\sigma,\tau}$. In fact, we know now that $z_{\tau,\eta}={z}|_{(-\infty,\tau]}\in E_{\rho+L_1\Gamma(1,\kappa_\Sigma),\tau}$ and $$\norm{z_{\tau,\eta}}_{E_{\rho+L_1\Gamma(1,\kappa_\Sigma),\tau}}\leq\Gamma(1,\kappa_\Sigma)|Q(\tau)\eta|_{N(\tau)}.$$ 
We remark that $\rho+L_1\Gamma(1,\kappa_\Sigma)$ may not satisfy the gap condition as can be seen in Remark~\ref{rem:MAXIMUM} below. 
\end{rem}

\begin{proof}[Proof of Theorem~\ref{thm:MAIN1}]
We define the manifold $\{\calm(\tau)\colon \tau\in\R\}$ by \eqref{e:MANIFOLD} using functions $z_{\tau,\eta}\in E_{\sigma,\tau}$ constructed in Subsection~\ref{subsec:ZET}. It is a graph of the function $\Sigma$ given by \eqref{e:DEFNSIGMA}, which is Lipschitz continuous w.r.t. the second variable with the Lipschitz constant $\kappa_\Sigma>0$ as shown in Corollary~\ref{cor:betterlip}. Characterization \eqref{eq-characterization-inertail} follows from Lemma \ref{lem:characterization}. \color{black} The manifold is 
invariant under the evolution process $\{T(t,\tau)\colon t\geq\tau\}$ by Lemma~\ref{lem:invariance}. Corollary~\ref{cor:POSINV} establishes the bound on the distance to the manifold with the exponent $\omega$ given in \eqref{e:DEFNOMEGA}. If $\omega>0$ then the manifold is inertial. By Lemma~\ref{lem:COMPLETETRAJECTORY}, through each point ${\eta}\in\calm(\tau)$ passes 
{a unique global solution of \eqref{e:VCF}} within the manifold and its growth in the past is controlled by the exponent $\rho+L_1\Gamma(1,\kappa_\Sigma)$.
\end{proof}

\begin{rem}\label{rem:MAXIMUM}
If $\Gamma$ is the maximum norm, the conditions and constants in Theorem~\ref{thm:MAIN1} are getting simpler.
We have already observed in Remark~\ref{rem:SHARPNESS} that if $\Gamma(\cdot)=\norm{\cdot}_\infty$, then the gap condition in Definition~\ref{def:gap} means that $\sigma\in(\rho+L_1,\gamma-L_2)$. Then we have $$\kappa=\kappa(\sigma)=\frac{L_2}{L_1}\frac{\sigma-\rho}{\gamma-\sigma}\in(\kappa_{\Sigma,1},\kappa_{\Sigma,2}),$$
where 
$$\kappa_{\Sigma,1}=\frac{L_2}{\gamma-\rho-L_1}<1\ \text{and}\ \kappa_{\Sigma,2}=\frac{\gamma-\rho-L_2}{L_1}>1$$
are the only solutions of~\eqref{e:KAPPABEST}, that is,
\begin{equation*}
\gamma-\rho=L_1\Gamma(1,\kappa_{\Sigma,i})+L_2\Gamma\left(\frac{1}{\kappa_{\Sigma,i}},1\right),\ i=1,2.
\end{equation*}
Thus the Lipschitz constant in \eqref{e:SIGMALIPBETTER} is then equal to 
$$\kappa_\Sigma=\kappa_{\Sigma,1}=\frac{L_2}{\gamma-\rho-L_1}<1.$$
Moreover, by \eqref{e:XIBACKWARDS}, for any $\tau\in\R$ and ${\eta}\in\calm(\tau)$ the backward part ${z}|_{(-\infty,\tau]}=z_{\tau,{\eta}}$ of the {global solution $z$ of \eqref{e:VCF}} through $\eta$ belongs to $E_{\rho+L_1,\tau}$ and {$\norm{z}_{E_{\rho+L_1,\tau}}\leq|Q(\tau)\eta|_{N(\tau)}$.}

Furthermore, in the presence of exponential dichotomy, that is, for $\gamma>0>\rho=-\gamma$, here the gap condition reduces  to $\gamma>\frac{L_1+L_2}{2}$. Then the manifold $\{\calm(t)\colon t\in\R\}$ is inertial with $\omega>0$ if $L_2\in\left(\frac{L_2}{2},\gamma-\frac{L_1}{2}\right)$, i.e., when $L_2$ is less than $\gamma-\frac{L_1}{2}$. 
\end{rem}

\begin{rem}
	If $\Gamma(a_1,a_2) = |a_1|+|a_2|$, then  the gap condition \eqref{e:GAPGENERAL00} is equivalent to say that $\gamma-\rho > (\sqrt{L_1}+\sqrt{L_2})^2$. In such case, in order for  \eqref{e:GAPGENERAL00} to hold, one needs to take
	$
	\sigma\in (\sigma_-,\sigma_+),$ where $$\sigma_- =  \frac{\gamma+\rho-L_2+L_1 - \sqrt{(L_1+L_2-(\gamma-\rho))^2-4L_1L_2}}{2}$$ and $$\sigma_+=\frac{\gamma+\rho-L_2+L_1 + \sqrt{(L_1+L_2-(\gamma-\rho))^2-4L_1L_2}}{2}.
	$$
	Starting from such $\sigma$, and taking $\kappa(\sigma) = \frac{L_2}{L_1}\frac{\sigma-\rho}{\gamma-\sigma}$ as the initial value of the iteration described in Corollary \ref{cor:betterlip}, we arrive at the Lipschitz constant
	 $$
	 \kappa_\Sigma = \frac{2L_2}{\gamma-\rho-(L_1+L_2)+\sqrt{(L_1+L_2-(\gamma-\rho))^2-4L_1L_2}}.
	 $$
	 If $L_1=L_2=L$, then this Lipschitz constant is equal to
	 $$
	 \kappa_\Sigma = \frac{\frac{\gamma-\rho}{L}-2-\sqrt{\left(\frac{\gamma-\rho}{L}-2\right)^2-4}}{2} = \frac{2}{\frac{\gamma-\rho}{L}-2+\sqrt{\left(\frac{\gamma-\rho}{L}-2\right)^2-4}}<1.
	 $$ 
	\end{rem}

\section{Stable manifold of the invariant Manifold}\label{sec:stable}

\par We now establish the existence of a stable manifold associated with the invariant manifold. This concept was introduced in \cite{CLMO-S25} as the complementary manifold of the inertial. Here, we prove its existence based on a refined gap condition and provide a characterization of this manifold in terms of solutions exhibiting specific forward growth due to the gap condition. The proof of the following result follows the structure of the proof of Theorem \ref{thm:MAIN1} and is therefore presented concisely. Note, however, that the obtained stable manifold $\{\mathcal{S}(t)\colon t\in \R\}$ is only positively invariant while the manifold $\{\calm(t)\colon t\in \R\}$ obtained in Theorem \ref{thm:MAIN1} is fully invariant.    

\begin{thm}\label{thm-stable-manifold}
Assume that the linear abstract Cauchy problem \eqref{e:LACP} generates a linear process $\{L(t,\tau)\colon t\geq\tau\}$ on a Banach space $X$, which has the exponential splitting (Definition \ref{def:splitting}) with projections $\{Q(\tau)\colon \tau\in\R\}$ and constants $\gamma>\rho\in\R$ and $M\geq 1$.
	
	Introducing the spaces $N(\tau)=Q(\tau)X$ and $S(\tau)=(I-Q(\tau))X$ for $\tau\in\R$, endowed with the norms {$|\cdot|_{N(\tau)}$} and $|\cdot|_{S(\tau)}$ given by \eqref{e:NORMNT} and \eqref{e:NORMST}, respectively,
	and equipping $X$ with the equivalent norms $\norm{\cdot}_{\tau}$ given in \eqref{e:MOVINGNORM} for a chosen admissible norm $\Gamma$ on $\R^2$ (Definition~\ref{defn:ADMISSIBLENORM}), we further assume that a continuous function $f\colon\R\times X\to X$ satisfies \eqref{e:ZEROCONDITION} and the uniform global Lipschitz conditions  
	\eqref{e:L1} and \eqref{e:L2} with constants $L_1,L_2>0$.
	
	Consider the perturbed abstract Cauchy problem \eqref{e:PERTURBED} and assume that it generates an evolution process $\{T(t,\tau)\colon t\geq\tau\}$ in $X$, which satisfies the variation of constants formula \eqref{e:VCF}.
	Finally, assume that {the gap condition \eqref{e:GAPGENERAL00} holds, that is,}
	\begin{equation*}
	\Gamma\left(\frac{L_1}{\sigma-\rho},\frac{L_2}{\gamma-\sigma}\right)<1\ \ \text{for some}\ \ \sigma\in (\rho,\gamma).
	\end{equation*}

    Then, there exists a family of sets $\{\mathcal{S}(t)\colon t\in \R\}$, the positively invariant stable manifold complementary to $\{ \calm(t)\colon t\in \R\}$, given by 
    \begin{equation}\label{eq-stable-manifold-characterization}
    \mathcal{S}(t)=\left\{\eta\in X\colon  \sup_{r\geq {t}} \{e^{\sigma r}\|T(r,t)\eta\|_r\}<\infty\right\},\ {t\in\R}.
\end{equation}
Moreover, $\mathcal{S}(t)=P_\Theta(t)X$, 
where {$P_\Theta(t)u=\Theta(t,u)+(I-Q(t))u$ for every $(t,u)\in\R\times X$, and} $\Theta\colon\mathbb{R}\times X \to X$ is a~function such that 
\begin{enumerate}
\item $\Theta(t,u)=\Theta(t,(I-Q(t))u)=Q(t)\Theta(t,u)$, 
and $\Theta(t, 0)=0$ for every $t\in\mathbb{R}$,
\item $|\Theta(t,\eta)-\Theta(t,\tilde{\eta})|_{N(t)}\leq \kappa_\Theta|(I-Q(t))(\eta-{\tilde{\eta}})|_{S(t)}$ for every $t\in \R$ {and $\eta,\tilde{\eta}\in X$ with} a~constant $0<\kappa_\Theta < \frac{L_1}{L_2}\frac{\gamma-\sigma}{\sigma-\rho}$ satisfying
\begin{equation}\label{e:KAPPATHETA}
\gamma-\rho=L_1\Gamma\left(1,\frac{1}{\kappa_\Theta}\right)+L_2\Gamma(\kappa_\Theta,1).
\end{equation}
\end{enumerate}
In addition, we have the following controlled growth estimate 
and 
\begin{equation}\label{DecayStM}
    \|T(t,\tau)P_{\Theta}(\tau)u\|_t \leq \frac{\Gamma(\kappa_\Theta,1)}{\Gamma(0,1)}e^{-(\gamma- L_2\Gamma(\kappa_\Theta,1))(t-\tau)}\|P_{\Theta}(\tau)u\|_\tau\ \ \text{for every}\ \ t\geq \tau, \ u\in X.
\end{equation}
\end{thm}

\begin{proof}
Existence and properties of  $\Theta$ are obtained using the argument which is complementary to the proof of Theorem \ref{thm:MAIN1} for existence of $\Sigma$.
First, choose $\sigma\in (\rho,\gamma)$ such that the gap condition \eqref{e:GAPGENERAL00} holds. Suppose that $\eta\in X$ and $\tau\in \R$ are such that $\sup_{t\geq \tau}\{e^{\sigma t}\|T(t,\tau)\eta\|_t\}$ is bounded.
Then, if $z(t)=T(t,\tau)\eta$ we denote $q(t):=Q(t)z(t) $ and $p(t):=(I-Q(t))z(t)$. Our aim is to construct the function  $\Theta\colon\mathbb{R}\times X \to X$ such that $q(t)=\Theta(t,p(t))$ for all $t\geq \tau$. The function $q$ satisfies the following Duhamel formula
\begin{equation*}
q(r)=L(r,t)q(t)+\int_{t}^{r}L(r,s)Q(s)f(s,z(s))ds\ \ \text{for}\ \ r\geq t\geq\tau.
\end{equation*}
Applying the inverse operator $L(t,r)$ to the above equation, we obtain
\begin{equation*}
q(t)=L(t,r)q(r)-\int_{t}^{r}L(t,s)Q(s)f(s,z(s))ds\ \ \text{for}\ \ r\geq t\geq\tau.
\end{equation*}
Since $e^{\sigma r}\|z(r)\|_r$ is bounded for $r\geq t$, we deduce that $e^{\rho r}\|z(r)\|_r\to 0$ as $r\to \infty$, whence
\begin{equation*}
    q(t)=-\int_t^{\infty} L(t,s)Q(s)f(s,z(s))ds,\  t\geq\tau.
\end{equation*}
Therefore, if we are able to find a continuous function $z\colon[\tau,+\infty)\to X$ satisfying 
\begin{equation*}
    z(t)=-\int_t^{\infty} L(t,s)Q(s)f(s,z(s))ds
+L(t,\tau)(I-Q(\tau))\eta +\int_{\tau}^{t} L(t,s)(I-Q(s))f(s,z(s))ds,
\end{equation*}
then, analogously to Lemma \ref{lem:lemma_solution}, {$z$ is a solution of \eqref{e:VCF} on $[\tau,\infty)$} with $(I-Q(\tau))z(\tau)={(I-Q(\tau))\eta}$, and the candidate for the sought function $\Theta$ is given as $\Theta(\tau,\eta):=Q(\tau)z(\tau)$. Then, if we define the nonlinear projection as $P_\Theta(\tau)\eta=\Theta(\tau,\eta)+(I-Q(\tau))\eta$, the image of this projection $P_\Theta(\tau) X$ is the complementary manifold of the invariant manifold $\mathcal{M}(\tau)$. 
Indeed, immediately, we get 
$\Theta(\tau,{(I-Q(\tau))\eta})=\Theta(\tau,\eta)=Q(\tau)\Theta(\tau,\eta)$. 

The function $z$ is constructed using a fixed point argument, which follows the lines of the argument of Subsection \ref{subsec:ZET}.
 Given $\tau\in \mathbb{R}$ and $\sigma\in (\rho,\gamma)$ such that \eqref{e:GAPGENERAL00} holds we define the space
\begin{equation*}
    F_{\sigma,\tau}=\{z\in C([\tau,{\infty});X)\colon  \|z\|_{F_{\sigma,\tau}}=\sup_{t\geq \tau} e^{\sigma (t-\tau)}\|z(t)\|_t<\infty\},
\end{equation*}
which is a Banach space with
the norm $\|\cdot\|_{F_{\sigma,\tau}}$.
Choosing $\eta\in X$, $z\in F_{\sigma,\tau}$ we define 
\begin{equation*}
(Hz)(t)=-\int_t^{\infty} L(t,s)Q(s)f(s,z(s))ds
 +L(t,\tau)(I-Q(\tau))\eta+\int_{\tau}^{t} L(t,s)(I-Q(s))f(s,z(s))ds
\end{equation*}
for $t\in [\tau,\infty)$. We prove that $H$ is a contraction from $F_{\sigma,\tau}$ on itself. To this end, we must estimate the quantity 
$$
\|(Hz)(t)\|_t=\Gamma(|Q(t)(Hz)(t)|_{N(t)},|(I-Q(t))(Hz)(t)|_{S(t)}),
$$
where 
$$
Q(t)(Hz)(t) = -\int_t^{\infty} L(t,s)Q(s)f(s,z(s))ds,
$$
and 
$$
(I-Q(t))(Hz)(t) = L(t,\tau)(I-Q(\tau))\eta+\int_{\tau}^{t} L(t,s)(I-Q(s))f(s,z(s))ds.
$$
We have
\begin{equation*}
       \left|\int_t^{\infty} L(t,s)Q(s)f(s,z(s))ds\right|_{N(t)}
   \leq L_1\int_t^{{\infty}} e^{\rho(s-t)}\|z(s)\|_s\, ds
   \leq \|z\|_{F_{\sigma,\tau}}L_1\int_t^{{\infty}} e^{\rho(s-t)-\sigma (s-\tau)}\, ds,
\end{equation*}
and
\begin{equation*}
\begin{split}
|(I-Q(t))(Hz)(t)|_{S(t)}&\leq |L(t,\tau)(I-Q(\tau))\eta|_{S(t)} +\int_\tau^t |L(t,s)(I-Q(s))f(s,z(s))|_{S(t)}ds\\
&\leq e^{\gamma(\tau-t)}|(I-Q(\tau))\eta|_{S(t)} + L_2\int_\tau^t  e^{\gamma(s-t)}\|z(s)\|_s\,ds\\
&\leq e^{\gamma(\tau-t)}|(I-Q(\tau))\eta|_{S(t)} +L_2\|z\|_{F_{\sigma,\tau}}\int_\tau^t  e^{\gamma(s-t)-\sigma (s-\tau)}\,ds.
\end{split}
\end{equation*}
Hence, for $t\geq \tau$, 
\begin{equation*}
        \begin{aligned}
            e^{\sigma (t-\tau)}\|(Hz)(t)\|_t&{\leq}\Gamma\bigg(\|z\|_{F_{\sigma,\tau}}\frac{L_1}{\sigma-\rho}, e^{(t-\tau)(\sigma-\gamma)}|{(I-Q(\tau))\eta}|_{S(\tau)}+\|z\|_{F_{\sigma,\tau}}\frac{L_2(1-e^{(\sigma-\gamma)(t-\tau)})}{\gamma-\sigma}\bigg),\\
            &\leq \Gamma\bigg(\|z\|_{F_{\sigma,\tau}}\frac{L_1}{\sigma-\rho}, |(I-Q(\tau))\eta|_{S(\tau)}+\|z\|_{F_{\sigma,\tau}}\frac{L_2}{\gamma-\sigma}\bigg).
        \end{aligned}
    \end{equation*}
{Therefore, we have} $Hz\in F_{\sigma,\tau}$.
Similar computations lead to 
\begin{equation*}
            e^{\sigma (t-\tau)}\|(Hz_1)(t)-(Hz_2)(t)\|_t\leq \|z_1-z_2\|_{F_{\sigma,\tau}}\Gamma\bigg(\frac{L_1}{\sigma-\rho},\frac{L_2}{\gamma-\sigma}\bigg)
    \end{equation*}
for $z_1,z_2\in F_{\sigma,\tau}$, and from the gap condition \eqref{e:GAPGENERAL00}, the operator $H$ has a unique fixed point $z^*\in F_{\sigma,\tau}$.
Define
\begin{equation*}
\Theta(\tau,\eta ) :=-\int_\tau^\infty L(\tau,s)Q(s)f(s,z^*(s)) ds\ \ \text{for every}\ \   (\tau,\eta)\in \R\times X.
\end{equation*}
An argument that follows the lines of {Lemma~\ref{lem:characterization}} 
implies that $P_\Theta(t) X = \mathcal{S}(t)$ for every $t\in \R$ with $\mathcal{S}(t)$ given by \eqref{eq-stable-manifold-characterization}, and the family $\{S(t)\colon t\in \R\}$ is positively invariant. 

To prove Lipschitzness, pick $\tau\in \R$, $\eta, \tilde \eta \in X$ and consider solutions $z, \tilde{z}:[\tau,\infty)\to X$ such that $z(\tau) = P_\Theta(\tau)\eta$ and $\tilde z(\tau) = P_\Theta(\tau)\tilde \eta$. Picking $\kappa=\frac{L_2}{L_1}\frac{\sigma-\rho}{\gamma-\sigma}$ as in \eqref{e:DEFNKAPPA}
we will 
study the function $\zeta(t):=|(I-Q(t))(z(t)-\tilde z(t))|_{S(t)}-\kappa|Q(t)(z(t)-\tilde z(t))|_{N(t)} = |v(t)|_{S(t)}-\kappa |u(t)|_{N(t)}$ and we will prove that $\zeta(\tau)\geq 0$, i.e., 
$$
|\Theta(\tau,\eta)-\Theta(\tau,\tilde  \eta)|_{N(t)} \leq \frac{1}{\kappa}|(I-Q(\tau))(\eta-\tilde \eta)|_{S(t)}.
$$
Without loss of generality we may assume that {$(I-Q(\tau))(\eta-\tilde \eta) \neq 0$}. Then, by Proposition \ref{prop:zeta_2} we only have to exclude the case where $\zeta(t)<0$ for every $t\in [\tau,\infty)$. We skip the argument to obtain this assertion as it is based on the standard estimate for the equation 
$$
u(t)=L(t,s)u(s)-\int_t^s L(t,r)Q(r)(f(r,z(r))-f(r,\tilde z(r))\, dr\ \ \text{for}\ \ \tau \leq t \leq s < \infty
$$
and closely follows the proof of Corollary \ref{cor:ALWAYSNEGATIVE}. 

Following the lines of the argument of {Theorem~\ref{thm:betterlip} and} Corollary \ref{cor:betterlip}, the Lipschitz constant can be refined, so that the Lipschitz constant $0<\kappa_\Theta<\frac{1}{\kappa}$ can be made to satisfy
\begin{equation*}
    \frac{L_1\Gamma(\kappa_\Theta,1)}{\gamma-\rho-L_2\Gamma(\kappa_\Theta,1)}=\kappa_\Theta.
\end{equation*}


    
Finally, we prove \eqref{DecayStM}. Let $u\in X$ and let $\eta = P_\Theta(\tau)u\in \mathcal{S}(\tau)$. 
    Denoting $p(t)=(I-Q(t))T(t,\tau)\eta$ for every $t\geq \tau$, we have the estimate
    \begin{equation*}
        |p(t)|_{S(t)}\leq e^{\gamma(\tau-t)}|(I-Q(\tau))\eta|_{{S(\tau)}}
        +L_2\Gamma(\kappa_\Theta,1)
        \int_\tau^t e^{\gamma(r-t)}|p(r)|_{S(r)} \, dr.
    \end{equation*}
    Applying the Gronwall inequality 
    to the last bound, we obtain 
     \begin{equation*}
|p(t)|_{S(t)}\leq e^{-(\gamma-L_2\Gamma(\kappa_\Theta,1))(t-\tau)}|(I-Q(\tau))\eta|_{S(\tau)}, \ t\geq \tau.
\end{equation*}
Since $\|T(t,\tau)\eta\|_t=\|P_\Theta(t)p(t)\|_t=
\Gamma(|\Theta(t,p(t))|_{N(t)},|p(t)|_{S(t)})\leq \Gamma(\kappa_\Theta,1)|p(t)|_{S(t)}$,
we obtain \eqref{DecayStM} and the proof is complete.
\end{proof}

The above result establishes the existence of a stable manifold associated with the zero solution. However, as noted in \cite{CLMO-S25}, it is also possible to prove the existence of stable manifolds associated with any global solution lying on the inertial manifold. More precisely, since $\mathcal{M}$ can be characterized as the union of global solutions, one can construct, for each such solution, a corresponding stable manifold, as we describe in the following definition.

\begin{defn} 
Under the conditions of Theorem  \ref{thm-stable-manifold}, there exists an invariant manifold $\{\mathcal{M}(t)\colon t\in\R\}$. 
Let $\xi\colon\R\to X$ be a global solution such that  
$\xi(t)\in \mathcal{M}(t)$ for every $t\in \R$.
The \emph{stable manifold of the invariant manifold along $\xi$} is a family $\{\mathcal{S}_\xi(t)\colon t\in\R\}$ such that 
	\begin{enumerate}
		\item $\mathcal{S}_\xi(t)$ is a Lipschitz manifold for each $t\in \mathbb{R}$;
		\item $\{\mathcal{S}_\xi(t)\colon t\in \mathbb{R}\}$ is positively invariant, i.e., $T(t,\tau)\mathcal{S}_\xi(\tau)\subset \mathcal{S}_\xi(t)$, for $t\geq \tau$;
		\item there exists $\delta\in (\rho,\gamma)$ such that, for every $\tau \in \R$ and $\eta \in S_\xi(\tau)$, there exists $K = K(\tau, \eta) > 0$ such that
  $$\|T(t,\tau) \eta-\xi(t)\|\leq Ke^{-\delta(t-\tau)}, \ t\geq \tau .$$ 
\end{enumerate}
\end{defn}

As described in \cite[Remark 2.4]{CLMO-S25}, it is possible to use Theorem \ref{thm-stable-manifold} to obtain the stable manifold along any global solution in the inertial manifold. 
Indeed, for a given global solution $\xi\colon\R\to X$, define 
$$\tilde{f}(t,v)=f(t,v+\xi(t))-f(t,\xi(t)), \ (t,v)\in \R\times X.$$
Then $\tilde{f}$ is globally Lipschitz and $\tilde{f}(t,0)=0$ for all $t\in \R$, and apply Theorem \ref{thm-stable-manifold} to obtain the stable manifold $\mathcal{S}_{0,\tilde{f}}$ along zero solution for $\tilde{f}$. 
Then the stable manifold $\mathcal{S}_{\xi,f}$ along $\xi$ for the problem with $f$ is obtained by $\mathcal{S}_{\xi,f}(t)=\xi(t)+\mathcal{S}_{0,\tilde{f}}(t)$ for all $t\in \R$.

\begin{rem}
Continuing the considerations of Remark~\ref{rem:MAXIMUM} when $\Gamma$ is the maximum norm on $\R^2$, we see that
$\kappa_{\Theta,i}=\frac{1}{\kappa_{\Sigma,i}}$, $i=1,2$, are the only solutions of~\eqref{e:KAPPATHETA}, that is,
\begin{equation*}
\gamma-\rho=L_1\Gamma\left(1,\frac{1}{\kappa_{\Theta,i}}\right)+L_2\Gamma\left(\kappa_{\Theta,i},1\right),\ i=1,2.
\end{equation*}
Since $\frac{1}{\kappa}\in(\kappa_{\Theta,2},\kappa_{\Theta,1})$, we deduce that 
the Lipschitz constant in Theorem~\ref{thm-stable-manifold} is then equal to 
$$\kappa_\Theta=\kappa_{\Theta,2}=\frac{L_1}{\gamma-\rho-L_2}<1.$$
Moreover, for any $\tau\in\R$ and $\eta\in\cals(\tau)$ the solution $z(t)=T(t,\tau)\eta$ for $t\geq\tau$ belongs to $F_{\gamma-L_2,\tau}$ and $\norm{z}_{F_{\gamma-L_2,\tau}}\leq|(I-Q(\tau))\eta|_{S(\tau)}$.
\end{rem}

\section{Differentiable invariant and inertial manifolds}\label{sec:diff}

We maintain the assumptions of the previous section. In this setting we prove in this section that the function $\Sigma(t,\eta)$ that defines the inertial manifold  is differentiable with respect to the variable $\eta$ provided that $f(t,\cdot)$ is Fr\'echet differentiable.
To this end, we additionally assume that $f\colon\R\times X\to X$ is Fr\'echet differentiable w.r.t. the second variable and we denote its Fr\'echet derivative $Df(t,u)\in\mathcal{L}(X,X)$, $(t,u)\in\R\times X$, that is,
\begin{equation}\label{e:FRECHET1}
	f(t,u+h)-f(t,u)-Df(t,u)h=\eps(t,u,h)\norm{h}_t\ \ \text{for}\ \ \ t\in\R,\ u,h\in X,
\end{equation}
where $\eps\colon\R\times X\times X\to X$ satisfies
\begin{equation}\label{e:FRECHET2}
	\eps(t,u,h)\to\eps(t,u,0)=0\text{ as }h\to 0\text{ for each }t\in\R,\ u\in X.
\end{equation}
The choice of the norm $\norm{\cdot}_t$ on the right-hand side of \eqref{e:FRECHET1} is for our convenience only as it is equivalent to $\norm{\cdot}$ by Lemma~\ref{lem:EQUIVNORMS}. 

\begin{thm}\label{thm:MAIN2}
Under assumptions of Theorem~\ref{thm:MAIN1} and \eqref{e:FRECHET1}, \eqref{e:FRECHET2}, the function $\Sigma\colon\R\times X\to X$, given in \eqref{e:DEFNSIGMA}, that defines the manifold $\{\calm(t)\colon t\in\R\}$, is differentiable w.r.t. the second variable and
\begin{equation}\label{e:DERSIGMA}
D_\eta\Sigma(\tau,\eta)=\int_{-\infty}^\tau L(\tau,s)(I-Q(s))Df(s,z_{\tau,\eta}(s))D_\eta[z_{\tau,\eta}(s)]ds,\ \tau\in\R,\ \eta\in X.
\end{equation}
\end{thm}

The above result will be proved in the following part of this section.

We first observe that from \eqref{e:L1}, \eqref{e:L2} and \eqref{e:FRECHET1}  it follows that for $t\in\R$ and $u,w\in X$ we have
\begin{equation}\label{e:L1DF}
	|Q(t)Df(t,u)w|_{N(t)}\leq L_1\norm{w}_{t},
\end{equation}
\begin{equation}\label{e:L2DF}
	|(I-Q(t))Df(t,u)w|_{S(t)}\leq L_2\norm{w}_{t}.
\end{equation}
Indeed, writing \eqref{e:FRECHET1} with $h=\Gamma w$ for an arbitrary $\Gamma>0$, we get by \eqref{e:L1} and \eqref{e:EQUIVNS}
$$\Gamma|Q(t)Df(t,u)w|_{N(t)}\leq\Gamma L_1\norm{w}_{t}+\Gamma M\norm{w}_t\norm{\eps(t,u,\Gamma w)}.$$
Dividing by $\Gamma$ and passing with $\Gamma\to 0$, we obtain \eqref{e:L1DF} by \eqref{e:FRECHET2}. The proof of \eqref{e:L2DF} is analogous. We also have
\begin{equation*}
	|Q(t)\eps(t,u,h)|_{N(t)}\leq 2L_1\ \ \text{for}\ \ t\in\R,\ u,h\in X,
\end{equation*}
\begin{equation*}
	|(I-Q(t))\eps(t,u,h)|_{S(t)}\leq 2L_2\ \ \text{for}\ \ \ t\in\R,\ u,h\in X.
\end{equation*}

Let us fix $\tau\in\R$. Given $\eta\in X$, under the conditions of Section~\ref{sec:LIP} with $\sigma\in(\rho,\gamma)$ satisfying the gap condition \eqref{e:GAPGENERAL00}, we showed that $H_{\tau,\eta}\colon E_{\sigma,\tau}\to E_{\sigma,\tau}$ defined by
\begin{equation*}
\begin{split}
H_{\tau,\eta}z(t)=L(t,\tau)Q(\tau)\eta&-\int_{t}^{\tau}L(t,s)Q(s)f(s,z(s))ds\\
&+\int_{-\infty}^{t}L(t,s)(I-Q(s))f(s,z(s))ds\ \ \text{for}\ \  t\in (-\infty,\tau],\ z\in E_{\sigma,\tau},
\end{split}
\end{equation*}
is a contraction on the Banach space 
$$E_{\sigma,\tau}=\{z\in C((-\infty,\tau];X)\colon \norm{z}_{E_{\sigma,\tau}}=\sup_{t\leq\tau}e^{\sigma(t-\tau)}\norm{z(t)}_{t}<\infty\},$$
hence it has a unique fixed point $z_{\tau,\eta}\in E_{\sigma,\tau}$, which defines the function $\psi_\tau\colon X\to E_{\sigma,\tau}$ by
$$\psi_\tau(\eta)=z_{\tau,\eta}\ \ \text{for}\ \ \eta\in X.$$
In the sequel we will drop the subscript $\tau$ in functions $\psi_\tau$, $z_{\tau,\eta}$ and the space $E_{\sigma,\tau}$, since we will be interested in the differentiability with respect to $\eta$.

Before we pass to the main result of this section we prove several preparatory estimates.
\begin{lem}
We have the following estimates valid for $\eta, h \in X$ and $t\leq \tau$:
\begin{align}\label{e:PSIONNT}
&|Q(t)[\psi(\eta+h)(t)-\psi(\eta)(t)]|_{N(t)}\leq|Q(\tau)h|_{N(\tau)}e^{(\rho+L_1\Gamma(1,\kappa_\Sigma))(\tau-t)},\\
&\label{e:PSIONST}
|(I-Q(t))[\psi(\eta+h)(t)-\psi(\eta)(t)]|_{S(t)}\leq\kappa_\Sigma|Q(\tau)h|_{N(\tau)}e^{(\rho+L_1\Gamma(1,\kappa_\Sigma))(\tau-t)},\\
&\label{e:CONTINUITYPSIT}
\norm{\psi(\eta+h)(t)-\psi(\eta)(t)}_{t}\leq\Gamma\left(1,\kappa_\Sigma\right)|Q(\tau)h|_{N(\tau)}e^{(\rho+L_1\Gamma(1,\kappa_\Sigma))(\tau-t)}.
\end{align}
\end{lem}

\begin{proof}
Estimating the difference of the two fixed point equations projected by $Q(t)$ and using the Lipschitz condition \eqref{e:SIGMALIPBETTER} for $\Sigma$ with $\kappa_\Sigma>0$, we deduce that
\begin{align*}
	& e^{\rho t}|Q(t)[\psi(\eta+h)(t)-\psi(\eta)(t)]|_{N(t)}\\
	& \ \ \leq e^{\rho\tau}|Q(\tau)h|_{N(\tau)}+L_1\Gamma(1,\kappa_\Sigma)\int_{t}^{\tau}e^{\rho s}|Q(s)[\psi(\eta+h)(s)-\psi(\eta)(s)]|_{N(s)}ds
	\ \text{for}\ t\leq\tau,\ \eta, h\in X,
	\end{align*}
whence by the Gronwall inequality we get \eqref{e:PSIONNT}. On the other hand, projecting the equation for the difference of two solutions by $I-Q(t)$ we have
\begin{align*}
	& e^{\gamma t}|(I-Q(t))[\psi(\eta+h)(t)-\psi(\eta)(t)]|_{S(t)}\leq L_2\int_{-\infty}^{t}e^{\gamma s}\norm{\psi(\eta+h)(s)-\psi(\eta)(s)}_{s}ds\\
	& \ \ \leq L_2\Gamma(1,\kappa_\Sigma)|Q(\tau)h|_{N(\tau)}e^{(\rho+L_1\Gamma(1,\kappa_\Sigma))\tau}\int_{-\infty}^{t}e^{(\gamma-\rho-L_1\Gamma(1,\kappa_\Sigma))s}ds\\
& \ \ =\frac{L_2\Gamma(1,\kappa_\Sigma)}{\gamma-\rho-L_1\Gamma(1,\kappa_\Sigma)}|Q(\tau)h|_{N(\tau)}e^{(\rho+L_1\Gamma(1,\kappa_\Sigma))\tau}e^{(\gamma-\rho-L_1\Gamma(1,\kappa_\Sigma))t}.
\end{align*}
This yields \eqref{e:PSIONST}, since by \eqref{e:KAPPABEST} we have
$\gamma-\rho-L_1\Gamma(1,\kappa_\Sigma)=\frac{L_2}{\kappa_\Sigma}\Gamma(1,\kappa_\Sigma).$
The last  estimate follows directly from \eqref{e:PSIONNT} and \eqref{e:PSIONST}.
\end{proof}

We consider the Banach space 
$$F_{\sigma}=\left\{Z\in C((-\infty,\tau];\mathcal{L}(X,X))\colon \norm{Z}_{F_{\sigma}}=\sup_{t\leq\tau}\sup_{\norm{w}\leq 1}e^{\sigma(t-\tau)}\norm{Z(t)w}_{t}=\sup_{\norm{w}\leq 1}\norm{Z(\cdot)w}_{E_\sigma}<\infty\right\}.$$
Note that if $Y$ is a linear and bounded map from $X$ to $E_\sigma$, then
$$\norm{Y}_{\mathcal{L}(X,E_\sigma)}=
\sup_{\norm{w}\leq 1}\norm{Yw}_{E_\sigma}=
\sup_{\norm{w}\leq 1}\sup_{t\leq\tau}e^{\sigma(t-\tau)}\norm{(Yw)(t)}_{t}\text{ for }Y\in \mathcal{L}(X,E_\sigma).$$
Thus, if $Z\in F_\sigma$ then $Y\colon X\to E_\sigma$ given by $(Yw)(t)=Z(t)w$ for $t\leq\tau$, $w\in X$, belongs to $\mathcal{L}(X,E_\sigma)$ and $\norm{Y}_{\mathcal{L}(X,E_\sigma)}=\norm{Z}_{F_\sigma}$.

Given $\eta\in X$  we define the mapping $J_\eta\colon F_\sigma\to F_\sigma$ by
\begin{equation*}
\begin{split}
(J_{\eta}Z)(t)&=L(t,\tau)Q(\tau)-\int_{t}^{\tau}L(t,s)Q(s)Df(s,z_\eta(s))Z(s)ds\\
&+\int_{-\infty}^{t}L(t,s)(I-Q(s))Df(s,z_\eta(s))Z(s)ds\ \ \text{for}\ \   Z\in F_\sigma,\ t\leq\tau.
\end{split}
\end{equation*}
We will prove that $J_\eta$ is well defined, that is, if $Z\in F_\sigma$, then also $J_\eta Z \in F_\sigma$, and that $J_\eta$ is a~contraction. 
To this end, we are in position to estimate
\begin{equation*}
\begin{split} 
|L(t,s)Q(s)Df(s,z_\eta(s))Z(s)w|_{N(t)}&\leq e^{\rho(s-t)}|Q(s)Df(s,z_\eta(s))Z(s)w|_{N(s)}\\
&\leq L_1 e^{\rho(s-t)}\norm{Z(s)w}_{s}\ \ \text{for}\ \ \ s\in[t,\tau],\ \norm{w}\leq 1,\\
\end{split} 
\end{equation*}
\begin{equation*}
\begin{split}
|L(t,s)(I-Q(s))Df(s,z_\eta(s))Z(s)w|_{S(t)}&\leq e^{\gamma(s-t)}|(I-Q(s))Df(s,z_\eta(s))Z(s)w|_{S(s)}\\
&\leq L_2 e^{\gamma(s-t)}\norm{Z(s)w}_{s}\ \ \text{for}\ \  s\leq t,\ \norm{w}\leq 1,
\end{split}
\end{equation*}
whence it follows that
\begin{align*}
&	e^{\sigma(t-\tau)}|Q(t)(J_\eta Z)(t)w|_{N(t)}\leq e^{(\sigma-\rho)(t-\tau)}|Q(\tau)w|_{N(\tau)}+L_1\norm{Z}_{F_\sigma}\int_{t}^{\tau}e^{(\sigma-\rho)(t-s)}ds\\
& \ \ \ =e^{(\sigma-\rho)(t-\tau)}|Q(\tau)w|_{N(\tau)}+L_1\norm{Z}_{F_\sigma}\frac{1-e^{(\sigma-\rho)(t-\tau)}}{\sigma-\rho}\ \ \text{for}\ \ \norm{w}\leq 1,\\
& e^{\sigma(t-\tau)}|(I-Q(t))(J_\eta Z)(t)w|_{S(t)}\leq L_2\norm{Z}_{F_\sigma}\int_{-\infty}^{t}e^{(\gamma-\sigma)(s-t)}ds=\frac{L_2}{\gamma-\sigma}\norm{Z}_{F_\sigma}\ \ \text{for}\ \ \norm{w}\leq 1.
\end{align*}
Consequently, we obtain by \eqref{e:QI-QONX}
$$e^{\sigma(t-\tau)}\norm{(J_\eta Z)(t)w}_{t}\leq M\Gamma(1,0)e^{(\sigma-\rho)(t-\tau)}+\Gamma\left(\frac{L_1}{\sigma-\rho},\frac{L_2}{\gamma-\sigma}\right)\norm{Z}_{F_\sigma}\ \ \text{for}\ \ t\leq\tau,\ \norm{w}\leq 1.$$
Thus $J_\eta$ is well defined and
\begin{equation*}
\norm{J_\eta Z}_{F_\sigma}\leq M\Gamma(1,0)+\Gamma\left(\frac{L_1}{\sigma-\rho},\frac{L_2}{\gamma-\sigma}\right)\norm{Z}_{F_\sigma},\ Z\in F_\sigma.
\end{equation*}
To prove that $J_\eta$ is a contraction, proceeding similarly with the estimates, we obtain
$$\norm{J_\eta Z-J_\eta\tilde Z}_{F_\sigma}\leq \Gamma\left(\frac{L_1}{\sigma-\rho},\frac{L_2}{\gamma-\sigma}\right)\norm{Z-\tilde Z}_{F_\sigma}\ \ \text{for}\ \ Z,\tilde Z\in F_\sigma.$$
By the Banach fixed point theorem for each $\eta\in X$ there exists a unique fixed point $Z_\eta\in F_\sigma$ of $J_\eta$.
We define $\Psi\colon X\to F_\sigma$ by
\begin{equation}\label{e:DEFPSI}
\Psi(\eta)=Z_{\eta}\ \ \text{for}\ \ \eta\in X.
\end{equation}
This function is a candidate for the  Fr\'echet derivative of $\psi$ as we demonstrate in the next theorem.

\begin{thm}\label{thm:DIFFPSI}
For each $t\leq\tau$ the function $\psi(\cdot)(t)\colon X\to X$ is  Fr\'echet differentiable and its Fr\'echet derivative w.r.t. $\eta\in X$ is given by $D_\eta[\psi(\eta)(t)]=\Psi(\eta)(t)\in\mathcal{L}(X,X)$ for $\eta\in X$.
More precisely, there exists a function $r\colon(-\infty,\tau]\times X\times X\to X$ satisfying $r(t,\eta,0)=0$ for $t\leq\tau$, $\eta\in X$ such that
\begin{equation}\label{e:RESTPSI}
\psi(\eta+h)(t)-\psi(\eta)(t)-\Psi(\eta)(t)h=r(t,\eta,h)\norm{h}\ \ \text{for every}\ \  t\leq\tau,\ \eta, h\in X,
\end{equation}
with
\begin{equation*}
\norm{r(t,\eta,h)}_{t}\to 0\text{ as }h\to 0\ \ \text{  for each  }\ \ t\leq\tau,\ \eta\in X.
\end{equation*}
\end{thm}

\begin{proof}
Note that for $\eta\in X$ and $h\neq 0$ we know that $r(\cdot,\eta,h)$ defined by \eqref{e:RESTPSI} belongs to $E_\sigma$. We will in fact show that
$$r(\cdot,\eta,h)\to 0\text{ in }E_\sigma\text{ as }h\to0.$$
For $h\neq 0$ we have
\begin{align*}
	& r(t,\eta,h)=-\norm{h}^{-1}\int_{t}^{\tau}L(t,s)Q(s)\eps(s,\psi(\eta)(s),\psi(\eta+h)(s)-\psi(\eta)(s))\norm{\psi(\eta+h)(s)-\psi(\eta)(s)}_{s}ds\\
& \ +\norm{h}^{-1}\int_{-\infty}^{t}L(t,s)(I-Q(s))\eps(s,\psi(\eta)(s),\psi(\eta+h)(s)-\psi(\eta)(s))\norm{\psi(\eta+h)(s)-\psi(\eta)(s)}_{s}ds\\
& \ -\int_{t}^{\tau}L(t,s)Q(s)Df(s,\psi(\eta)(s))r(s,\eta,h)ds+\int_{-\infty}^{t}L(t,s)(I-Q(s))Df(s,\psi(\eta)(s))r(s,\eta,h)ds.
\end{align*}
Denote
$$\eps_1(s,\eta,h)=\eps(s,\psi(\eta)(s),\psi(\eta+h)(s)-\psi(\eta)(s)).$$
Using \eqref{e:L1DF}, \eqref{e:L2DF} and the bound \eqref{e:CONTINUITYPSIT} combined with \eqref{e:QI-QONX}, we get for $t\leq\tau$
\begin{align*}
&	e^{\sigma(t-\tau)}|Q(t)r(t,\eta,h)|_{N(t)}\leq \frac{L_1}{\sigma-\rho}\norm{r(\cdot,\eta,h)}_{E_\sigma}\\
& \qquad  +M\Gamma(1,\kappa_\Sigma)e^{(\sigma-\rho)(t-\tau)}\int_{t}^{\tau}e^{L_1\Gamma(1,\kappa_\Sigma)(\tau-s)}|Q(s)\eps_1(s,\eta,h)|_{N(s)}ds,\\
& e^{\sigma(t-\tau)}|(I-Q(t))r(t,\eta,h)|_{S(t)}\leq\frac{L_2}{\gamma-\sigma}\norm{r(\cdot,\eta,h)}_{E_\sigma}\\
& \qquad   +M\Gamma(1,\kappa_\Sigma)e^{(\gamma-\sigma)(\tau-t)}\int_{-\infty}^{t}e^{(\gamma-\rho-L_1\Gamma(1,\kappa_\Sigma))(s-\tau)}|(I-Q(s))\eps_1(s,\eta,h)|_{S(s)}ds.
\end{align*}
Thus we obtain
\begin{align*}
	& e^{\sigma(t-\tau)}\norm{r(t,\eta,h)}_{t}\leq \Gamma\left(\frac{L_1}{\sigma-\rho},\frac{L_2}{\gamma-\sigma}\right)\norm{r(\cdot,\eta,h)}_{E_\sigma}\\
	&\qquad \ \ +M\Gamma(1,\kappa_\Sigma)\Gamma\Bigg(e^{(\sigma-\rho)(t-\tau)}\int_{t}^{\tau}e^{L_1\Gamma(1,\kappa_\Sigma)(\tau-s)}|Q(s)\eps_1(s,\eta,h)|_{N(s)}ds,\\
&\qquad \qquad \ \  e^{(\gamma-\sigma)(\tau-t)}\int_{-\infty}^{t}e^{(\gamma-\rho-L_1\Gamma(1,\kappa_\Sigma))(s-\tau)}|(I-Q(s))\eps_1(s,\eta,h)|_{S(s)}ds\Bigg).
\end{align*}
Therefore, we have
$$\norm{r(\cdot,\eta,h)}_{E_\sigma}\leq\Gamma\left(\frac{L_1}{\sigma-\rho},\frac{L_2}{\gamma-\sigma}\right)\norm{r(\cdot,\eta,h)}_{E_\sigma}+M\Gamma(1,\kappa_\Sigma)\Gamma(A(\eta,h),B(\eta,h)),$$
where
\begin{align*}
	& A(\eta,h)=\sup_{t\leq\tau}e^{(\sigma-\rho)(t-\tau)}\int_{t}^{\tau}e^{L_1\Gamma(1,\kappa_\Sigma)(\tau-s)}|Q(s)\eps_1(s,\eta,h)|_{N(s)}ds,\\
	& B(\eta,h)=\sup_{t\leq\tau}e^{(\gamma-\sigma)(\tau-t)}\int_{-\infty}^{t}e^{(\gamma-\rho-L_1\Gamma(1,\kappa_\Sigma))(s-\tau)}|(I-Q(s))\eps_1(s,\eta,h)|_{S(s)}ds.
\end{align*}
Since $\Gamma(\frac{L_1}{\sigma-\rho},\frac{L_2}{\gamma-\sigma})<1$, we will prove the claim if we show that
$$(i)\quad A(\eta,h)\to 0\text{ as }h\to 0,$$
$$(ii)\quad B(\eta,h)\to 0\text{ as }h\to 0.$$
First note that $A(\eta,h)$ and $B(\eta,h)$ are well defined. Indeed, we have for $t\leq\tau$ and $\eta,h\in X$
\begin{align*}
&e^{(\sigma-\rho)(t-\tau)}\int_{t}^{\tau}e^{L_1\Gamma(1,\kappa_\Sigma)(\tau-s)}|Q(s)\eps_1(s,\eta,h)|_{N(s)}ds
\leq 2L_1e^{(\sigma-\rho)(t-\tau)}\int_{t}^{\tau}e^{L_1\Gamma(1,\kappa_\Sigma)(\tau-s)}ds\\
&\qquad\leq \frac{2}{\Gamma(1,\kappa_\Sigma)}e^{(\sigma-\rho-L_1\Gamma(1,\kappa_\Sigma))(t-\tau)},\\
&e^{(\gamma-\sigma)(\tau-t)}\int_{-\infty}^{t}e^{(\gamma-\rho-L_1\Gamma(1,\kappa_\Sigma))(s-\tau)}|(I-Q(s))\eps_1(s,\eta,h)|_{S(s)}ds\\
&\qquad \leq 2L_2e^{(\gamma-\sigma)(\tau-t)}\int_{-\infty}^{t}e^{(\gamma-\rho-L_1\Gamma(1,\kappa_\Sigma))(s-\tau)}ds=\frac{2L_2}{\gamma-\rho-L_1\Gamma(1,\kappa_\Sigma)}e^{(\sigma-\rho-L_1\Gamma(1,\kappa_\Sigma))(t-\tau)},
\end{align*}
since 
$$\sigma-\rho-L_1\Gamma(1,\kappa_\Sigma)>\sigma-\rho-L_1\Gamma(1,\kappa)=\sigma-\rho-(\sigma-\rho)\Gamma\left(\frac{L_1}{\sigma-\rho},\frac{L_2}{\gamma-\sigma}\right)>0.$$
Thus we have
$$A(\eta,h)\leq\frac{2}{\Gamma(1,\kappa_\Sigma)}\ \text{ and }\ B(\eta,h)\leq \frac{2L_2}{\gamma-\rho-L_1\Gamma(1,\kappa_\Sigma)}.$$
Now, contrary to the required claim, suppose that $A(\eta,h)\not\to 0$ as $h\to 0$. Thus there exists $\epsilon_0>0$ and $h_n\to 0$ such that $A(\eta,h_n)>\epsilon_0$ for all $n\in\N$. Therefore, there exist $t_n\leq\tau$ such that 
$$\epsilon_0<e^{(\sigma-\rho)(t_n-\tau)}\int_{t_n}^{\tau}e^{L_1\Gamma(1,\kappa_\Sigma)(\tau-s)}|Q(s)\eps_1(s,\eta,h_n)|_{N(s)}ds.$$
Since the right-hand side is bounded from above by $\frac{2}{\Gamma(1,\kappa_\Sigma)}e^{(\sigma-\rho-L_1\Gamma(1,\kappa_\Sigma))(t_n-\tau)}$,
the sequence $t_n$ cannot have a subsequence tending to $-\infty$. Hence there exists $\tau_0$ such that $t_n\geq\tau_0$ for all $n\in\N$.
We have
$$\epsilon_0<\int_{\tau_0}^{\tau}|Q(s)\eps_1(s,\eta,h_n)|_{N(s)}ds.$$
By~\eqref{e:FRECHET2} and \eqref{e:CONTINUITYPSIT}, for each $s\in[\tau_0,\tau]$ and $\eta\in X$ we have $|Q(s)\eps_1(s,\eta,h_n)|_{N(s)}\to 0$ as $n\to\infty$ and $|Q(s)\eps_1(s,\eta,h_n)|_{N(s)}\leq 2L_1$, which is integrable on $[\tau_0,\tau]$. Thus, the Lebesgue dominated convergence theorem implies that the right-hand side converges to zero as $n\to\infty$. This is a~contradiction with the choice of $\epsilon_0$.

Similarly, contrary to the claim, suppose now that $B(\eta,h)\not\to 0$ as $h\to 0$. Thus there exists $\epsilon_0>0$ and $h_n\to 0$ such that $B(\eta,h_n)>\epsilon_0$ for all $n\in\N$. Therefore, there exist $t_n\leq\tau$ such that 
$$\epsilon_0<e^{(\gamma-\sigma)(\tau-t_n)}\int_{-\infty}^{t_n}e^{(\gamma-\rho-L_1\Gamma(1,\kappa_\Sigma))(s-\tau)}|(I-Q(s))\eps_1(s,\eta,h_n)|_{S(s)}ds.$$
Since $$(\gamma-\sigma)(\tau-t_n)+(\gamma-\rho-L_1\Gamma(1,\kappa_\Sigma))(s-\tau)=(\sigma-\rho-L_1\Gamma(1,\kappa_\Sigma))(s-\tau)+(\gamma-\sigma)(s-t_n),$$
we get
$$\epsilon_0<\int_{-\infty}^{\tau}e^{(\gamma-\rho-L_1\Gamma(1,\kappa_\Sigma))(s-\tau)}|(I-Q(s))\eps_1(s,\eta,h_n)|_{S(s)}ds.$$
For each $s\in(-\infty,\tau]$ and $\eta\in X$ we have $e^{(\gamma-\rho-L_1\Gamma(1,\kappa_\Sigma))(s-\tau)}|(I-Q(s))\eps_1(s,\eta,h_n)|_{S(s)}\to 0$ as $n\to\infty$ and $$e^{(\gamma-\rho-L_1\Gamma(1,\kappa_\Sigma))(s-\tau)}|(I-Q(s))\eps_1(s,\eta,h_n)|_{S(s)}\leq 2L_2e^{(\gamma-\rho-L_1\Gamma(1,\kappa_\Sigma))(s-\tau)},$$ 
which is integrable on $(-\infty,\tau]$. Thus the Lebesgue dominated convergence theorem shows that the right-hand side converges to zero as $n\to\infty$. This is a contradiction with the choice of $\epsilon_0$.
\end{proof}

Since $\psi(\eta)(\tau)=z_{\eta}(\tau) = Q(\tau)\eta+\Sigma(\tau,\eta)$, the differentiability of $\psi$ implies differentiability of $\Sigma$ with respect to $\eta$ variable. 

\begin{proof}[Proof of Theorem~\ref{thm:MAIN2}]
By Theorem~\ref{thm:DIFFPSI}, $\Sigma(\tau,\eta)=\psi(\eta)(\tau)-Q(\tau)\eta$ is differentiable w.r.t. $\eta$ and
\begin{equation}\label{e:DERIVFORMULA}
D_\eta\Sigma(\tau,\eta)=\Psi(\eta)(\tau)-Q(\tau),
\end{equation}
where $\Psi(\eta)(\tau)=Z_{\eta}(\tau)$ and
$$Z_\eta(\tau)=Q(\tau)+\int_{-\infty}^{\tau}L(\tau,s)(I-Q(s))Df(s,z_\eta(s))Z_{\eta}(s)ds.$$
Since $Z_{\eta}(s)=D_\eta[z_{\eta}(s)]$ for $s\leq\tau$, we obtain \eqref{e:DERSIGMA}.
\end{proof}

\section{\texorpdfstring{$C^1$}\ \  invariant and inertial manifolds}\label{sec:c1}
In addition to the previously assumed gap condition \eqref{e:GAPGENERAL00}, i.e.,
$$\Gamma\left(\frac{L_1}{\sigma-\rho},\frac{L_2}{\gamma-\sigma}\right)<1,$$
and the Fr\'echet differentiability \eqref{e:FRECHET1}, \eqref{e:FRECHET2} of the function $f\colon \R\times X\to X$, 
in this section we further assume that $Df(t,\cdot)$ is continuous, i.e.,
\begin{equation}\label{e:FRECHET3}
\norm{Df(t,u)-Df(t,u_0)}_{\mathcal{L}(X,X)}\to 0\text{ as }u\to u_0\text{ in }X\text{ for every }t\in\R,\ u_0\in X. 
\end{equation}
The goal of this section is to show $C^1$ smoothness of the inertial manifold $\{\calm(t)\colon t\in\R\}$ defined via the function $\Sigma$ in Theorem~\ref{thm:MAIN1}.

\begin{thm}\label{thm:MAIN3}
Under assumptions of Theorem~\ref{thm:MAIN2} and \eqref{e:FRECHET3}, the derivative of $\Sigma$ w.r.t. $\eta$ variable is continuous, i.e., for each $\tau\in\R$ and $\eta\in X$
$$D_\eta\Sigma(\tau,\eta+h)\to D_\eta\Sigma(\tau,\eta)\ \text{in}\ \mathcal{L}(X,X)\ \text{as}\ h\to 0.$$
\end{thm}

Keeping the notation of Section~\ref{sec:diff}, we first show the continuity of function $\Psi$ defined in \eqref{e:DEFPSI}, which will be sufficient for Theorem~\ref{thm:MAIN3}.
For that purpose, note that by the continuity of $\Gamma$ we can choose $\sigma<\mu<\gamma$ such that 
\begin{equation*}
\Gamma\left(\frac{L_1}{\mu-\rho},\frac{L_2}{\gamma-\mu}\right)<1.
\end{equation*}

\begin{thm}
The function $X\ni \eta\mapsto \Psi(\eta)\in F_\sigma\subset F_\mu$ defined in \eqref{e:DEFPSI} is continuous in $F_\mu$, i.e.,
$$\norm{\Psi(\eta+h)-\Psi(\eta)}_{F_\mu}\to 0\text{ as }h\to0\text{ for each }\eta\in X.$$ 
In particular, we have
\begin{equation}\label{e:CONTINFSIGMA}
\norm{\Psi(\eta+h)(t)-\Psi(\eta)(t)}_{\mathcal{L}(X,X)}\to0\text{ as }h\to0\ \text{for each}\ t\leq\tau\ \text{and}\ \eta\in X.
\end{equation}
\end{thm}

\begin{proof}
We first observe that $\Psi$ defined in \eqref{e:DEFPSI} satisfies
\begin{align*}
	& \Psi(\eta+h)(t)-\Psi(\eta)(t)=-\int_{t}^{\tau}L(t,s)Q(s)Df(s,\psi(\eta+h)(s))\Psi(\eta+h)(s)ds\\
	&\ \ \ \  +\int_{t}^{\tau}L(t,s)Q(s)Df(s,\psi(\eta)(s))\Psi(\eta)(s)ds\\
	&\ \ \ \  +\int_{-\infty}^{t}L(t,s)(I-Q(s))Df(s,\psi(\eta+h)(s))\Psi(\eta+h)(s)ds\\
& \ \ \ \  -\int_{-\infty}^{t}L(t,s)(I-Q(s))Df(s,\psi(\eta)(s))\Psi(\eta)(s)ds\ \ \text{for}\ \ t\leq\tau,\ \eta, h\in X.
\end{align*}
Thus, we have
\begin{align*}
	& \Psi(\eta+h)(t)-\Psi(\eta)(t)=-\int_{t}^{\tau}L(t,s)Q(s)[Df(s,\psi(\eta+h)(s))-Df(s,\psi(\eta)(s))]\Psi(\eta)(s)ds\\
	& \ \ -\int_{t}^{\tau}L(t,s)Q(s)Df(s,\psi(\eta+h)(s))[\Psi(\eta+h)(s)-\Psi(\eta)(s)]ds\\
& \ \ +\int_{-\infty}^{t}L(t,s)(I-Q(s))[Df(s,\psi(\eta+h)(s))-Df(s,\psi(\eta)(s))]\Psi(\eta)(s)ds\\
& \ \ +\int_{-\infty}^{t}L(t,s)(I-Q(s))Df(s,\psi(\eta+h)(s))[\Psi(\eta+h)(s)-\Psi(\eta)(s)]ds\ \ \text{for}\ \ t\leq\tau,\ \eta, h\in X.
\end{align*}
Hence, by \eqref{e:L1DF}, for $\norm{w}\leq 1$ we have 
\begin{align*}
&	|Q(t)[\Psi(\eta+h)(t)-\Psi(\eta)(t)]w|_{N(t)}\leq L_1\int_{t}^{\tau}e^{\rho(s-t)}\norm{[\Psi(\eta+h)(s)-\Psi(\eta)(s)]w}_{s}ds \\
& \ \  +\int_{t}^{\tau}e^{\rho(s-t)}|Q(s)[Df(s,\psi(\eta+h)(s))-Df(s,\psi(\eta)(s))]\Psi(\eta)(s)w|_{N(s)}ds.
\end{align*}
Similarly, by \eqref{e:L2DF}, for $\norm{w}\leq 1$ we have 
\begin{align*}
	&|(I-Q(t))[\Psi(\eta+h)(t)-\Psi(\eta)(t)]w|_{S(t)}\leq L_2\int_{-\infty}^{t}e^{\gamma(s-t)}\norm{[\Psi(\eta+h)(s)-\Psi(\eta)(s)]w}_{s}ds\\
	& \ \ +\int_{-\infty}^{t}e^{\gamma(s-t)}|(I-Q(s))[Df(s,\psi(\eta+h)(s))-Df(s,\psi(\eta)(s))]\Psi(\eta)(s)w|_{S(s)}ds.\end{align*}
Thus, we get
\begin{align*}
	&e^{\mu(t-\tau)}|Q(t)[\Psi(\eta+h)(t)-\Psi(\eta)(t)]w|_{N(t)}\leq L_1\norm{\Psi(\eta+h)-\Psi(\eta)}_{F_\mu}\int_{t}^{\tau}e^{(\mu-\rho)(t-s)}ds\\
	& \ \ +e^{\mu(t-\tau)}\int_{t}^{\tau}e^{\rho(s-t)}|Q(s)[Df(s,\psi(\eta+h)(s))-Df(s,\psi(\eta)(s))]\Psi(\eta)(s)w|_{N(s)}ds,\\
& e^{\mu(t-\tau)}|(I-Q(t))[\Psi(\eta+h)(t)-\Psi(\eta)(t)]w|_{S(t)}\leq L_2\norm{\Psi(\eta+h)-\Psi(\eta)}_{F_\mu}\int_{-\infty}^{t}e^{(\gamma-\mu)(s-t)}ds\\& \ \ +e^{\mu(t-\tau)}\int_{-\infty}^{t}e^{\gamma(s-t)}|(I-Q(s))[Df(s,\psi(\eta+h)(s))-Df(s,\psi(\eta)(s))]\Psi(\eta)(s)w|_{S(s)}ds.\end{align*}
Therefore, we obtain
$$\norm{\Psi(\eta+h)-\Psi(\eta)}_{F_\mu}\leq\Gamma\left(\frac{L_1}{\mu-\rho},\frac{L_2}{\gamma-\mu}\right)\norm{\Psi(\eta+h)-\Psi(\eta)}_{F_\mu}+\Gamma(A_1(\eta,h),B_1(\eta,h)),$$
where
\begin{align*}
	&A_1(\eta,h) =\sup_{t\leq\tau}\sup_{\norm{w}\leq 1}\Big(e^{(\mu-\rho)(t-\tau)}\cdot\\
	& \qquad \qquad  \qquad  \cdot \int_{t}^{\tau}e^{\rho(s-\tau)}|Q(s)[Df(s,\psi(\eta+h)(s))-Df(s,\psi(\eta)(s))]\Psi(\eta)(s)w|_{N(s)}ds\Big),\\
	& B_1(\eta,h) =\sup_{t\leq\tau}\sup_{\norm{w}\leq 1}\Big(e^{(\gamma-\mu)(\tau-t)}\cdot\\
	& \qquad \qquad  \qquad  \cdot\int_{-\infty}^{t}e^{\gamma(s-\tau)}|(I-Q(s))[Df(s,\psi(\eta+h)(s))-Df(s,\psi(\eta)(s))]\Psi(\eta)(s)w|_{S(s)}ds\Big).\end{align*}
Since $\Gamma\left(\frac{L_1}{\mu-\rho},\frac{L_2}{\gamma-\mu}\right)<1$, we will prove the claim if we show that
$${A_1(\eta,h)\to 0\ \text{ and }\ B_1(\eta,h)\to 0\ \text{ as }\ h\to 0.}$$
To this end, first note that $A_1(\eta,h)$ and $B_1(\eta,h)$ are well defined. Indeed, by \eqref{e:L1DF}, \eqref{e:L2DF} we have
for $t\leq\tau$ and $\norm{w}\leq 1$
\begin{align*}
&	e^{(\mu-\rho)(t-\tau)}\int_{t}^{\tau}e^{\rho(s-\tau)}|Q(s)[Df(s,\psi(\eta+h)(s))-Df(s,\psi(\eta)(s))]\Psi(\eta)(s)w|_{N(s)}ds\\
	& \ \ \ \ \leq2L_1\norm{\Psi(\eta)}_{F_\sigma}e^{(\mu-\rho)(t-\tau)}\int_{t}^{\tau}e^{(\sigma-\rho)(\tau-s)}ds\leq\frac{2L_1}{\sigma-\rho}\norm{\Psi(\eta)}_{F_\sigma}e^{(\mu-\sigma)(t-\tau)},\\
& e^{(\gamma-\mu)(\tau-t)}\int_{-\infty}^{t}e^{\gamma(s-\tau)}|(I-Q(s))[Df(s,\psi(\eta+h)(s))-Df(s,\psi(\eta)(s))]\Psi(\eta)(s)w|_{S(s)}ds\\
& \ \ \ \ \leq 2L_2\norm{\Psi(\eta)}_{F_\sigma}e^{(\gamma-\mu)(\tau-t)}\int_{-\infty}^{t}e^{(\gamma-\sigma)(s-\tau)}ds=\frac{2L_2}{\gamma-\sigma}\norm{\Psi(\eta)}_{F_\sigma}e^{(\mu-\sigma)(t-\tau)}.
\end{align*}
Thus we have
$$A_1(\eta,h)\leq\frac{2L_1}{\sigma-\rho}\norm{\Psi(\eta)}_{F_\sigma}\ \text{ and }\ B_1(\eta,h)\leq\frac{2L_2}{\gamma-\sigma}\norm{\Psi(\eta)}_{F_\sigma}.$$
Suppose, contrary to the claim, that $A_1(\eta,h)\not\to 0$ as $h\to 0$. Thus there exist $\epsilon_0>0$ and $h_n\to 0$ such that $A_1(\eta,h_n)>\epsilon_0$ for all $n\in\N$. Therefore, there exist $t_n\leq\tau$ and $\norm{w_n}\leq 1$ such that 
$$\epsilon_0<e^{(\mu-\rho)(t_n-\tau)}\int_{t_n}^{\tau}e^{\rho(s-\tau)}|Q(s)[Df(s,\psi(\eta+h_n)(s))-Df(s,\psi(\eta)(s))]\Psi(\eta)(s)w_n|_{N(s)}ds.$$
Since the right-hand side is bounded from above by $\frac{2L_1}{\sigma-\rho}\norm{\Psi(\eta)}_{F_\sigma}e^{(\mu-\sigma)(t_n-\tau)}$ and $\mu>\sigma$,
the sequence $t_n$ cannot have a subsequence tending to $-\infty$. Hence there exists $\tau_0$ such that $t_n\geq\tau_0$ for all $n\in\N$.
We obtain
\begin{equation}\label{e:Q100}
\epsilon_0<\int_{\tau_0}^{\tau}e^{\mu(s-\tau)}|Q(s)[Df(s,\psi(\eta+h_n)(s))-Df(s,\psi(\eta)(s))]\Psi(\eta)(s)w_n|_{N(s)}ds.
\end{equation}
For each  $s\in[\tau_0,\tau]$ we have 
$$e^{\mu(s-\tau)}|Q(s)[Df(s,\psi(\eta+h_n)(s))-Df(s,\psi(\eta)(s))]\Psi(\eta)(s)w_n|_{N(s)}\leq 2L_1\norm{\Psi(\eta)}_{F_\mu},$$
which is integrable on $[\tau_0,\tau]$.

Now, recalling \eqref{e:QI-QONX} 
and deducing from Lemma~\ref{lem:EQUIVNORMS} that for $G\in\mathcal{L}(X,X)$
$$\norm{Gx}\leq c_\Gamma\norm{G}_{\mathcal{L}(X,X)}\norm{x}_{s},\ s\in\R,\ x\in X,$$
we further have
\begin{align*}
&e^{\mu(s-\tau)}|Q(s)[Df(s,\psi(\eta+h_n)(s))-Df(s,\psi(\eta)(s))]\Psi(\eta)(s)w_n|_{N(s)}\\
	&\qquad \qquad  \leq Mc_\Gamma\norm{Df(s,\psi(\eta+h_n)(s))-Df(s,\psi(\eta)(s))}_{\mathcal{L}(X,X)}\norm{\Psi(\eta)}_{F_\mu}.
\end{align*}
By \eqref{e:CONTINUITYPSIT} and \eqref{e:FRECHET3} this shows that the integrand on the right-hand side of \eqref{e:Q100} tends to zero as $n\to\infty$ pointwise on $[\tau_0,\tau]$. 
Thus the Lebesgue dominated convergence theorem shows that the right-hand side of \eqref{e:Q100} converges to zero as $n\to\infty$. This is a contradiction with the choice of $\epsilon_0$.

In a similar way, suppose now, contrary to the claim, that $B_1(\eta,h)\not\to 0$ as $h\to 0$. Thus there exist $\epsilon_0>0$ and $h_n\to 0$ such that $B_1(\eta,h_n)>\epsilon_0$ for all $n\in\N$. Therefore, there exist $t_n\leq\tau$ and $\norm{w_n}\leq 1$ such that 
$$\epsilon_0<e^{(\gamma-\mu)(\tau-t_n)}\int_{-\infty}^{t_n}e^{\gamma(s-\tau)}|(I-Q(s))[Df(s,\psi(\eta+h_n)(s))-Df(s,\psi(\eta)(s))]\Psi(\eta)(s)w_n|_{S(s)}ds.$$
Since $(\gamma-\mu)(\tau-t_n)+\gamma(s-\tau)=\mu(s-\tau)+(\gamma-\mu)(s-t_n)$, we get
\begin{equation}\label{e:Q101}
\epsilon_0<\int_{-\infty}^{\tau}e^{\mu(s-\tau)}|(I-Q(s))[Df(s,\psi(\eta+h_n)(s))-Df(s,\psi(\eta)(s))]\Psi(\eta)(s)w_n|_{S(s)}ds.
\end{equation}
Now, for every $s\in(-\infty,\tau]$ we have 
$$e^{\mu(s-\tau)}|(I-Q(s))[Df(s,\psi(\eta+h_n)(s))-Df(s,\psi(\eta)(s))]\Psi(\eta)(s)w_n|_{S(s)}\leq 2L_2\norm{\Psi(\eta)}_{F_\sigma}e^{(\mu-\sigma)(s-\tau)},$$
which is an integrable function on $(-\infty,\tau]$, since $\mu>\sigma$. Moreover, we get 
\begin{align*}
	&e^{\mu(s-\tau)}|(I-Q(s))[Df(s,\psi(\eta+h_n)(s))-Df(s,\psi(\eta)(s))]\Psi(\eta)(s)w_n|_{S(s)}\\
	& \ \ \leq Mc_\Gamma\norm{Df(s,\psi(\eta+h_n)(s))-Df(s,\psi(\eta)(s))}_{\mathcal{L}(X,X)}\norm{\Psi(\eta)}_{F_\mu}.
\end{align*}
By \eqref{e:CONTINUITYPSIT} and \eqref{e:FRECHET3} this shows that the integrand on the right-hand side of \eqref{e:Q101} tends to zero as $n\to\infty$ pointwise on $(-\infty,\tau]$. 
Thus the Lebesgue dominated convergence theorem shows that the right-hand side of \eqref{e:Q101} tends to zero as $n\to\infty$. This is a contradiction with the choice of $\epsilon_0$.
\end{proof}

\begin{proof}[Proof of Theorem~\ref{thm:MAIN3}]
Since in the proof of Theorem~\ref{thm:MAIN2} we have shown \eqref{e:DERIVFORMULA}, i.e.,
$$D_\eta\Sigma(\tau,\eta)=\Psi(\eta)(\tau)-Q(\tau),\ \tau\in\R,\ \eta\in X,$$
the continuity of $D_\eta\Sigma(\tau,\cdot)$ in $\mathcal{L}(X,X)$ follows directly from \eqref{e:CONTINFSIGMA}.
\end{proof}

\end{document}